\documentclass[a4paper]{article}
\usepackage[utf8x]{inputenc}
\usepackage{amsmath, amsfonts,amssymb,mathtools,amsthm,mathrsfs,bbm}
\usepackage{bbding, pifont}
\usepackage{graphicx}
\usepackage{xcolor}
\usepackage{natbib}  
\usepackage{bbding}
\newcommand{\cmark}{\ding{51}}%
\newcommand{\xmark}{\ding{55}}%
\usepackage[rightcaption]{sidecap}
\usepackage{csquotes}
\usepackage{tikz}
\usepackage{hyperref}

\definecolor{darkgrey}{rgb}{0.75,0.75,0.75}

\usepackage[rightcaption]{sidecap}
\newtheorem{theorem}{Theorem}
\newtheorem{proposition}{Proposition}[section]
\newtheorem{lemma}[proposition]{Lemma}

\newtheorem{definition}[proposition]{Definition}
\theoremstyle{definition}
\newtheorem{remark}[proposition]{Remark}

\makeatletter
\renewcommand{\@fnsymbol}[1]{\ensuremath{%
   \ifcase#1\or 1\or 2\or 3\or
   \mathsection\or \mathparagraph\or \|\or 1\or
   2\or 3 \else\@ctrerr\fi}}
\makeatother

\numberwithin{equation}{section}

\begin{document}
\title{\Large The infinitely many genes model with horizontal gene
  transfer} \author{Franz Baumdicker\thanks{Abteilung f\"ur
    Mathematische Stochastik, Albert-Ludwigs University of Freiburg,
    Eckerstr.~1, D--79104 Freiburg, Germany}\; and Peter
  Pfaffelhuber$\mbox{}^{1,}$\thanks{e-mail:
    p.p@stochastik.uni-freiburg.de}}

\thispagestyle{empty}

\date{\today} 

\maketitle

\begin{abstract}
  \noindent
  The genome of bacterial species is much more flexible than
  that of eukaryotes. In particular, the distributed genome hypothesis
  for bacteria states that the total number of genes present in a
  bacterial population is greater than the genome of every single
  individual. The \emph{pangenome}, i.e.\ the set of all genes of a
  bacterial species (or a sample), comprises the core genes which are
  present in all living individuals, and accessory genes, which are
  carried only by some individuals. Bacteria have developed mechanisms
  in order to exchange genes horizontally, i.e.\ without a direct
  relationship.
  Here, we extend the \emph{infinitely many genes model} from
  Baumdicker, Hess and Pfaffelhuber (2010) for such horizontal gene
  transfer. We take a genealogical view and give a construction --
  called the \emph{Ancestral Gene Transfer Graph} -- of the joint
  genealogy of all genes in the pangenome. As application, we compute
  moments of several statistics (e.g.\ the number of differences
  between two individuals and the gene frequency spectrum) under the
  infinitely many genes model with horizontal gene transfer.
\end{abstract}


~


\section{Introduction}
Many prokaryotic species (i.e.\ bacteria and archea) are now known to
have highly flexible genomes (e.g.\
\citealp{Tettelin2005,Ehrlich2005,Tettelin2008,Koonin:2012:Front-Cell-Infect-Microbiol:22993722}).
Unlike in eukaryotes, genes can be transferred horizontally (i.e.\
without a direct relationship between donor and recipient) between
prokaryotic individuals of either different or the same population.
As a result, gene content can differ substantially between strains
from the same population. For example, the pathogenic strain
\emph{E. coli} O157:H7 carries 1387 genes which are absent in the
commensal strain \emph{E. coli} K-12 \citep{Perna2001}. This huge
variation in gene content led to the concepts of the \emph{distributed
  genome} of bacteria and their \emph{pangenome}
\citep{Tettelin2005,Ehrlich2005}. 

In order to understand the growing amount of genomic data from
bacterial species, classical population genetic theory -- using
mutation, selection, recombination and genetic drift as main
evolutionary forces -- must be extended in order to include realistic
mechanisms of horizontal gene transfer (HGT). Since genomic data from
prokaryotic species has become abundant only recently, HGT can in
particular be seen as a newly discovered evolutionary factor
\citep{Doolittle:1999:Trends-Cell-Biol:10611671,Koonin:1997:Mol-Microbiol:9379893}.
However, theoretical work on the population genomics of HGT is still
in its infancy. In order to include HGT in population genetic models,
some scenarios have to be distinguished according to the following
basal mechanisms: Transformation, which is the uptake of genetic
material from the environment. Transduction, which describes the
infection of a prokaryote by a lysogenic virus (phage) which provides
additional genetic material that can be integrated into the bacterial
genome. Conjugation, which is also termed \emph{bacterial sex}, which
requires a direct link (pilus) between two bacterial cells and leads
to exchange of genetic material. In addition, small virus-like
elements called Gene Transfer Agents (GTAs) have been found which may
become even more important for the amount of horizontal genetic
exchange in some species \citep{pmid20929803}. Another mechanism of
horizontal gene transfer are due to mobile genetic elements like
plasmids, gene cassettes and transposons, which transfer genes even
within a single individual
\citep{delaCruz:2000:Trends-Microbiol:10707066}. Although all these
mechanisms transfer only parts of the gene sequences, it is a valid
approach to model only complete gene transfers events, since bacteria
are efficient in getting rid of non-functional genetic material.  Most
importantly, when considering horizontal gene transfer by
transformation, transduction and conjugation, transformation (and to a
lesser extend also transduction) transfers genes mainly between
distantly related species, while conjugation only works for bacteria
from closely related species.

~

In \cite{BaumdickerHessPfaffelhuber2010} (see also
\citealp{BaumdickerHessPfaffelhuber2011}), we presented the
\emph{infinitely many genes model}, a population genomic model which
includes HGT from different species, e.g.\ by transformation, but no
HGT within species. It accounts for \emph{gene gain} (or gene uptake)
from the environment (at rate $\theta/2$) along the genealogical tree
which describes the relationships between the individuals of the
population. The term \emph{gene gain} covers HGT from other species as
well as gene genesis (the formation of new genes), because from the
perspective of a species under consideration these two mechanisms are
indistinguishable.  Pseudogenization may lead to deletion of genes and
is incorporated by \emph{gene loss} (at rate $\rho/2$). The model uses
the coalescent \citep{Kingman1982,Hudson1983} as underlying genealogy
instead of a fixed (phylogenetic) tree. On the latter, the same two
mechanisms were studied already by
\cite{Huson:2004:Bioinformatics:15044248}. 

In the present paper, we extend the infinitely many genes model in
order to incorporate events of intraspecies horizontal gene
transfer. We stress that HGT in bacteria differs from crossover
recombination in eukaryotes, since only single, non-homologous, genes
are horizontally transferred in bacteria, while only homologous
genomic regions are transferred by recombination in
eukaryotes. Accordingly, since we aim at a genealogical picture of HGT
in bacteria, the ancestral recombination graph
\citep{Hudson1983,GriffithsMarjoram1997} as an extension of the
coalescent, cannot be used. Rather, we model HGT such that each gene
present in the population comes with its own events of HGT, resulting
in the Ancestral Gene Transfer Graph (AGTG). In the limit of large
population sizes, we compute moments of several quantities of
interest. The gene frequency spectrum -- see Theorem~\ref{T1} --
describes the amount of genes present in $k$ out of $n$
individuals. In \cite{BaumdickerHessPfaffelhuber2011} the gene
frequency spectrum has been used to test whether a bacterial
population shows unusual patterns for neutral evolution.  In
Theorem~\ref{T2}, we give our results for the expectations of the
average number of genes per individual, and the average number of
symmetric pairwise differences and the total number of genes,
respectively. Calculations which give the variances of some of these
quantities, can be carried out using the AGTG and are given in
Theorem~\ref{T3}.

~

The paper is organized as follows: In Section~\ref{S:model}, we
introduce the infinitely many genes model with horizontal gene
transfer. After stating our results in Section~\ref{S:results}, we
discuss our results with a view towards biological applications in
Section~\ref{sec:disc}. In Section~\ref{S:AGTG} we introduce our main
tool, the AGTG. The proofs of the main results, Theorems \ref{T1}--\ref{T3}, are given in
Section~\ref{S:proofs}.

\section{The model}
\label{S:model}
We introduce two different views on the same model. In this section,
we describe a Moran model forwards in time, including events of gene
gain, gene loss and horizontal transfer of genes; see also
Figure~\ref{graphicmoran} for a graphical representation. Later, in
Section~\ref{S:AGTG} we describe how to obtain the distribution of
genes in equilibrium using a genealogy-based approach.

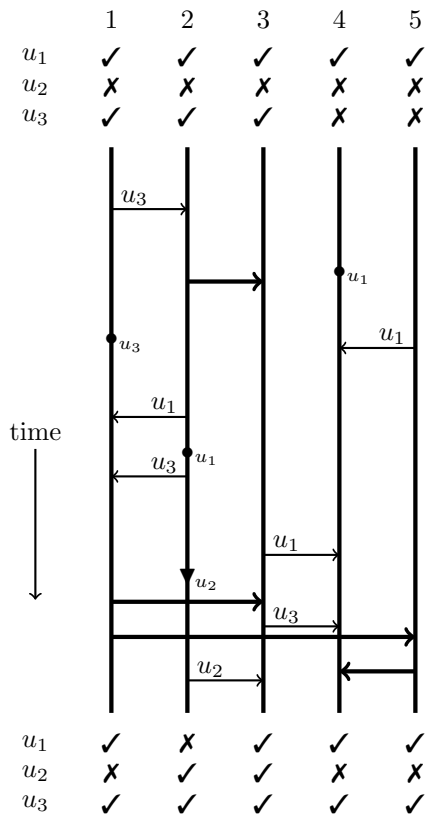
\begin{SCfigure}
\centering
\pgfmathsetseed{111206}
\begin{tikzpicture}
    \draw[thick, <- ] (0,1.5) --(0,3.5);
    \draw[white, <- ] (6,1.5) --(6,3.5);
    \node[above] at (0,3.5){time};
    \foreach \x in {1,...,5}{
    \draw[ultra thick] (\x,0)--(\x,5*1.5);
    \node at (\x,9.2){\x};};
    \foreach \x in {1,...,5}
    {
    \pgfmathsetmacro{\Xa}{random(1,5)}
    \pgfmathsetmacro{\Xb}{random(1,5)}
    \pgfmathsetmacro{\Ya}{5*random()*1.5}
    \ifx\Xa\Xb \else\draw [->,ultra thick, black] (\Xa,\Ya) -- (\Xb,\Ya);\fi
    };
\foreach \y in {1,...,3}
{
     \pgfmathdeclarerandomlist{rndcolor}{{black}}
     \pgfmathrandomitem{\mycol}{rndcolor}
     \pgfmathrandomitem{\mycolt}{rndcolor}
     \pgfmathsetmacro{\dark}{random(30,60)}
     \foreach \x in {1,...,3}
     {
       \pgfmathtruncatemacro{\Xa}{random(1,5)}  
    \pgfmathtruncatemacro{\Xb}{random(1,5)}
    \pgfmathsetmacro{\Ya}{5*random()*1.5}
    \ifx \Xa\Xb 
     \else
	\draw [->, \mycol, thick] (\Xa,\Ya) -- (\Xb,\Ya); 
	\ifnum\Xa<\Xb
	\node[right] at (\Xa,\Ya + 0.15){$u_\y$};
        \else
	\node[left] at (\Xa,\Ya + 0.15){$u_\y$};
        \fi
    \fi
    };
};
%
\node[scale=1.] at (2,2.3*1.5){$\bullet$};
 \node at (2.1,2.3*1.5){$\phantom{X}_{u_1}$};
\node[scale=1.] at (1,3.3*1.5){$\bullet$};
 \node at (1.1,3.3*1.5){$\phantom{X}_{u_3}$};
\node[scale=1.] at (4,3.9*1.5){$\bullet$};
 \node at (4.1,3.9*1.5){$\phantom{X}_{u_1}$};
 \node at (2,1.2*1.5){$\blacktriangledown$};
 \node at (2.1,1.2*1.5){$\phantom{X}_{u_2}$};

    \node at (0,-0.4){$u_1$};
    \node at (0,-0.8){$u_2$};
    \node at (0,-1.2){$u_3$};
    \node at (1,-0.4){\cmark};
    \node at (1,-0.8){\xmark};
    \node at (1,-1.2){\cmark};
    \node at (2,-0.4){\xmark};
    \node at (2,-0.8){\cmark};
    \node at (2,-1.2){\cmark};
    \node at (3,-0.4){\cmark};
    \node at (3,-0.8){\cmark};
    \node at (3,-1.2){\cmark};
    \node at (4,-0.4){\cmark};
    \node at (4,-0.8){\xmark};
    \node at (4,-1.2){\cmark};
    \node at (5,-0.4){\cmark};
    \node at (5,-0.8){\xmark};
    \node at (5,-1.2){\cmark};

    \node at (0,8.7){$u_1$};
    \node at (0,8.3){$u_2$};
    \node at (0,7.9){$u_3$};
    \node at (1,8.7){\cmark};
    \node at (1,8.3){\xmark};
    \node at (1,7.9){\cmark};
    \node at (2,8.7){\cmark};
    \node at (2,8.3){\xmark};
    \node at (2,7.9){\cmark};
    \node at (3,8.7){\cmark};
    \node at (3,8.3){\xmark};
    \node at (3,7.9){\cmark};
    \node at (4,8.7){\cmark};
    \node at (4,8.3){\xmark};
    \node at (4,7.9){\xmark};
    \node at (5,8.7){\cmark};
    \node at (5,8.3){\xmark};
    \node at (5,7.9){\xmark};
\end{tikzpicture}
\caption{\label{graphicmoran} The graphical representation of the
  Moran model of size $N=5$ from Definition \ref{def:moran}.  At thick
  arrows, the individual at the tip of the arrow is replaced by a copy
  of the individual at the tail. Three mechanisms are illustrated as
  follows: \newline 1. \emph{Gene loss} of gene $u$ is given at events
  $\bullet_u$; rate $\rho/2$ per gene per line.\newline 2. \emph{Gene
    gain} of gene $u$ is given at events $\blacktriangledown_u$; rate
  $\theta/2$ per line.\newline 3. \emph{Horizontal gene transfer} of
  gene $u$ from individual $i$ to $j$ is given through a thin arrow $i
  \!\!\xrightarrow{\mbox{ }u\mbox{ }} \!\!j$; rate $\gamma/(2N)$ per
  gene for every ordered pair $(i,j)$, indicating a potential HGT
  event.\newline Here we show only the events for genes $u_1,u_2$ and
  $u_3$. Presence and absence of these genes at the top gives rise
  through resampling, gene gain (only gene $u_2$), gene loss and HGT
  to presence and absence of the genes at the bottom.}
\end{SCfigure}

We consider the following model for bacterial evolution:
Each
bacterial cell carries a set of \emph{genes} and every gene belongs
either to the \emph{core genome} or to the \emph{accessory
  genome}. The uncountable set $I:=[0,1]$ is the space of conceivable
accessory genes.
In addition there is a set of persistent genes, the core genome -- see
also \cite{BaumdickerHessPfaffelhuber2011}.
As by definition these core gens can never be lost or gained and are
just present in all individuals we will ignore these genes in the
following analysis. A population of constant size consists of $N$
individuals, where each individual represents a (genome of a) bacterial cell which consists of several accessory genes.
 We model this accessory genome of
individual $i$ at time $t$ by a finite counting measure $\mathcal
G_i^N(t)$ on $I$. We will identify finite counting measures with the
set of atoms, i.e.\ we write $u\in \mathcal G_i^N(t)$ if $\langle
\mathcal G_i^N(t), 1_u\rangle \geq 1$. The dynamics of the model is
such that $\langle \mathcal G_i^N(t), 1_u\rangle \leq 1$ for all $i$
and $u\in[0,1]$, almost surely. In other words, there is at most one
copy of each gene in any individual.


The population evolves according to Moran dynamics; see also
Figure~\ref{graphicmoran}. That is, time is continuous and every
(unordered) pair of individuals $\{i,j\}$ undergoes a resampling event
at rate~1. Here, in each resampling event between individuals~$i$
and~$j$, one bacterium is chosen at random ($i$, say), produces one
offspring which replaces the other individual~($j$ in this case) such
that the population size stays constant. The offspring carries the
same genes as the parent, i.e.\ if an offspring of~$i$ replaces~$j$ at
time~$t$, we have $\mathcal G_j^N(t) = \mathcal G_i^N(t-)$. In
addition to such resampling events, the following (independent) events
occur:

\begin{enumerate}
\item \emph{Gene loss}: For gene $u\in\mathcal G_i^N(t-)$ in
  individual~$i$, at rate~$\rho/2$, we have $\mathcal G_i^N(t) =
  \mathcal G_i^N(t-)\setminus\{u\}$, i.e.\ gene~$u$ is lost
  from~$\mathcal G_i^N(t)$.
\item \emph{Gene gain}: For every individual $i$, at rate $\theta/2$,
  choose $U$ uniform in $[0,1]$ and set $\mathcal G_i^N(t) = \mathcal
  G_i^N(t-)\cup\{U\}$, i.e.\ every individual gains an (almost surely)
  new gene at rate $\theta/2$.
\item \emph{Horizontal gene transfer}: For every (ordered) pair of
  individuals $(i,j)$ and $u\in\mathcal G_i^N(t)$, a horizontal gene
  transfer event occurs at rate $\gamma/(2N)$. For such an event, set
  $\mathcal G_j^N(t) = \mathcal G_j^N(t-)\cup\{u\}$ and $\mathcal
  G_i^N(t) = \mathcal G_i^N(t-) $, i.e.\ individual~$i$ is the
  \emph{donor} of gene~$u$ and transfers a copy of the gene $u$ to the
  \emph{recipient}~$j$.
\end{enumerate}
Horizontal gene transfer events can as well be written in the
measure-valued notation as $\mathcal G_j^N(t) = (\mathcal G_j^N(t-) +
\delta_u)\wedge 1$. The '$\wedge 1$'-term indicates that we do not
model paralogous genes, i.e.\ horizontal gene transfer events have no
effect if the recipient individual~$j$ already carries the transferred
gene.

\begin{definition}[Moran model with horizontal gene transfer]
  \label{def:moran}
  \sloppy We refer to $(\mathcal G_1^N(t), \dots, \mathcal
  G_N^N(t))_{t\geq 0}$ undergoing the above dynamics as the
  \emph{Moran model for bacterial genomes with horizontal gene flow}.
\end{definition}

\begin{lemma}[Equilibrium]
  \label{rem:equi}
  The Moran model of size $N$ for bacterial genomes with horizontal
  gene flow has a unique mixing, ergodic equilibrium. We denote random
  finite measures distributed according to this equilibrium by
  $\mathcal G_1^N:=\mathcal G_1^N(\infty),...,\mathcal G_N^N:=\mathcal
  G_N^N(\infty)$.
\end{lemma}

\begin{proof}
  First, existence of a stationary measure follows from tightness of
  the family $(\mathcal G_1^N(t),...,\mathcal G_N^N(t))_{t\geq 0}$. In
  order to see this, note that a single gene in frequency $k$ rises to
  $k+1$ at rate $(\tfrac 12+\tfrac \gamma {2N}) k(N-k)$ and decreases
  to $k-1$ at rate $\tfrac 12 k(N-k) - \tfrac \rho 2 k$. Denoting the
  hitting time of~0 of this birth-death process by $T$, we have that
  $\mathbb E_k[T]<\infty$ by positive recurrence of the birth-death
  process for all $k=1,...,N$. As a consequence, we can bound $\mathbb
  E\Big[\sum_{n=1}^N \langle \mathcal G_n^N(t), 1\rangle\Big]$ by
  contributions from the time-0 population and newly gained genes,
  i.e.\ by
  $$\sup_{t\geq 0}\mathbb
  E\Big[\sum_{n=1}^N \langle \mathcal G_n^N(t), 1\rangle\Big] \leq
  \sum_{n=1}^N \langle \mathcal G_n^N(0), 1\rangle + N^2 \tfrac \theta
  2 \sup_{t\geq 0} \int_0^t \mathbb P_1(T>t) dt < \infty,$$ which is
  enough for \sloppy tightness; see \cite{Kallenberg2002}, Lemma
  14.15. Any weak limit must be a stationary measure by standard
  arguments. Now, we show that $(\mathcal G_1^N(t),...,\mathcal
  G_N^N(t))_{t\geq 0}$, started in any stationary measure at time
  $-\infty$, is mixing. Indeed, there is a random, finite time $S$
  when all genes present at time $t=0$ have become lost. Now, the
  distribution of $(\mathcal G_1^N(t),...,\mathcal G_N^N(t))_{t\leq
    0}$ is independent of $(\mathcal G_1^N(t),...,\mathcal
  G_N^N(t))_{t\geq S}$, since the latter only depends on events in the
  graphical construction after $t\geq 0$.
\end{proof}

\begin{remark}[Diffusion limit]
  {\em In mathematical population genetics, one frequently studies
    models of finite populations forward in time, constructs their
    diffusion limit -- most often a Fleming--Viot measure-valued
    diffusion -- and only afterwards uses genealogical relationships
    in order to have a dual process to the Fleming--Viot
    measure-valued diffusion and to compute specific properties of the
    underlying forwards model.  Since our interest in the present
    paper lies in seeing the effects of horizontal gene transfer on
    summary statistics (see Theorems~\ref{T1}--\ref{T3}), we take
    another route here and leave the construction of the infinite
    model forwards in time for future research. Here, one would have
    to consider the set of counting measures on $[0,1]$ as a type
    space (which is locally compact), and define the current state of
    a finite population as the empirical measure of types on this
    state space. Constructing the diffusion limit then gives the
    measure-valued diffusion. We foresee two challenges in such a
    construction: (i) The corresponding {\it recombination operator}
    modeling HGT events for the limiting Fleming-Viot process is
    unbounded, since the number of genes in a genome is unbounded;
    (ii) Although we will give a genealogical construction of HGT in
    Section~\ref{S:AGTG}, it is not straight-forward to interpret the
    resulting graph as a dual process. }
\end{remark}

\begin{remark}[Gene transfer of more than one gene]
  {\em Note that we model only the transfer of DNA segments too small for
  more than one gene, although it is known that the transfer of
  multiple genes at once can occur \citep{Price2008}.  However, we
  postulate that the results of Theorem~\ref{T1} and Theorem~\ref{T2}
  are as well valid for a model with multiple gene transfers and
  multiple gene losses, where at rate $\gamma'$ ($\rho'$) each gene of
  an individual is transfered (lost) with some probability $p_\gamma$
  ($p_\rho$).
  In this case, first moments of the statistics we compute in
  Theorems~\ref{T1} and~\ref{T2}, as well as Lemma~\ref{l:diff}, are
  not affected if we replace $\gamma$ and $\rho$ by $\gamma' p_\gamma$
  and $\rho' p_\rho$.  However, our results for second moments in
  Theorem~\ref{T3} differ in the case of multiple gene transfer/loss
  events.}
\end{remark}

~

\noindent
We are mainly interested in large populations. The corresponding limit
is usually referred to as large population limit in the population
genetic literature. The following result of the Moran model
with HGT will already be useful in various applications.

\begin{lemma}[The frequency path of a single gene]
  \label{l:diff}
  Let $X^N(t)$ be the frequency of gene $u$ at time $t$ in the Moran
  model for bacterial genomes with horizontal gene flow of size $N$
  with $X^N(0)$ such that $X^N(0)\xrightarrow{N\to\infty}x$. Then, in
  the large population limit, $N\to\infty$, the process
  $(X^N(t))_{t\geq 0}$ converges weakly to the solution of the SDE
  \begin{align}\label{eq:SDE}
    dX = \big(-\tfrac\rho 2 X + \tfrac \gamma 2 X(1-X) \big)dt +
    \sqrt{X(1-X)} dW
  \end{align}
  with $X(0)=x$ for some Brownian motion $W$.
\end{lemma}

\begin{remark}[The diffusion~\eqref{eq:SDE} in population genetics]
  {\em The diffusion~\eqref{eq:SDE} also appears in population genetics
  models including selection (see e.g.\ \cite{Kimura1964},
  \cite[chapt. 5.3]{Ewens2004}, \cite[chapt. 7.2]{Durrett2008}). In
  the present setting, the term proportional to $X(1-X)dt$ appears
  because horizontal gene flow increases the frequency of the gene by
  a rate which is proportional to the number of possible
  donor/recipient-pairs of individuals; see
  also~\cite{Bonhoeffer2013}.

  Due to the close connection of horizontal gene transfer with
  selective models, a comparison to recent work is appropriate. In
  particular, the theory for the frequency spectrum in selective
  models with irreversible mutations is carried out in
  \cite{Fisher1930,Wright1938,Kimura1964,Kimura1969}. We re-derive
  these results in our proof of Theorem~\ref{T1} below, but we stress
  that the genealogical interpretation we give is derived with a
  special focus on horizontal gene flow, but not to the selective
  case.}
\end{remark}

\begin{proof}[Proof of Lemma~\ref{l:diff}]
  As in the proof of Lemma~\ref{rem:equi}, note that gene loss reduces
  $X^N$ with rate $\tfrac \rho 2 N X^N$. Second, horizontal gene
  transfer increases $X^N$ with rate $\tfrac{\gamma}{2N}N^2
  X^N(1-X^N)$. By construction, the evolution of frequencies of gene
  $u$ is a Markov process with generator
  \begin{align*}
    (G^Nf)(x) & = N(N-1)x(1-x)\Big(\frac 12 f(x+1/N) + \frac 12
    f(x-1/N) - f(x)\Big) \\ & \qquad - \frac{\rho Nx}{2}
    (f(x-1/N)-f(x)) + \frac{\gamma N^2x(1-x)}{2N}(f(x+1/N)-f(x)) \\ &
    \xrightarrow{N\to\infty} \frac 12 x(1-x) f''(x) +(- \frac\rho 2 x
    + \gamma x(1-x))f'(x)
  \end{align*}
  for $f\in\mathcal C^2([0,1])$. Using e.g.\ standard results
  from~\cite[chapt. 4]{Ewens2004} it is now easy to show weak
  convergence to the diffusion~\eqref{eq:SDE}.
\end{proof}

\section{Results on Summary Statistics}
\label{S:results}
Consider a sample $\mathcal G_1^N,\dots,\mathcal G_n^N$ of size $n$
taken from the Moran model of size $N$ in equilibrium. We introduce
several statistics under the above dynamics:
\begin{itemize}
\item The \emph{average number of genes (in the accessory genome)} of the sampled $n$ individuals is
  given by
  \begin{align}\label{eq:Hbar}
    A^{(n)} := A^{(n,N)} := \dfrac 1n \sum_{i=1}^n |\mathcal G_i^N|
  \end{align}
  where $ |\mathcal G_i^N| := \langle \mathcal G_i^N, 1\rangle$ is the
  total number of accessory genes in individual~$i$.
\item The \emph{average number of symmetric pairwise differences} is given by
  \begin{align}\label{eq:Dbar}
    D^{(n)} := D^{(n,N)} := \dfrac 1{n(n-1)} \sum_{1\leq i \neq j\leq n}
    |\mathcal G_i^N \setminus \mathcal G_j^N|
  \end{align}
  where $ \mathcal G_i^N \setminus \mathcal G_j^N := (\mathcal G_i^N -
  \mathcal G_j^N)^+$ are the genes present in $i$ but not in $j$.
\item The \emph{size of the accessory genome} of the sample is given by
  \begin{align}\label{eq:G}
    G^{(n)}:=G^{(n,N)}:=\Big| \bigcup_{i=1}^n \mathcal G_i^N \Big|
  \end{align}
  where $ \bigcup_{i=1}^n \mathcal G_i^N = \Big(\sum_{i=1}^n \mathcal
  G_i^N\Big) \wedge 1$ is the set of genes present in any individual
  from the sample, counting each gene only once no matter in how many individuals it is present.
\item The \emph{gene frequency spectrum (of the accessory genome)} is
  given by $G_1^{(n)}:=G^{(n,N)}_1,\dots,G_n^{(n)}:=G^{(n,N)}_n$,
  where
  \begin{align}\label{eq:Gk}
    G_k^{(n)} := G_k^{(n,N)} := |\{u\in I: u\in\mathcal G_i^N \text{
      for exactly } k \text{ different }i\}|.
  \end{align}
\end{itemize}

\begin{remark}[Notation]
  {\em In the following results, we will suppress the superscript~$N$ of
  the population size. Instead, we query that our results hold in the
  \emph{large population limit}. E.g.\ if we say that~\eqref{eq:7}
  holds in the large population limit, we really mean that
  $$ \mathbb E[A^{(n,N)}] \xrightarrow{N\to\infty} \frac{\theta}{\rho} \left( 1 +
    \sum\limits_{m=1}^{\infty} \frac{\gamma^m}{(1+\rho)_{m \uparrow}}
  \right).$$}
\end{remark}

~

\noindent
The proofs of all results presented here are given in
Section~\ref{S:proofs}. For first moments, we provide proofs using
diffusion theory and Lemma~\ref{l:diff}. For second moments, we rely
on the Ancestral Gene Transfer Graph (AGTG) of
Section~\ref{S:AGTG}. Since the proofs of the results are either using
Lemma~\ref{l:diff} or the AGTG or both, we formulate the following
three Theorems. 

\begin{theorem}[Gene frequency spectrum]
  \label{T1} Consider a sample of size $n$ taken from the Moran model
  for bacterial genomes with horizontal gene flow with $\rho>0,
  \theta>0, \gamma\geq 0$ in equilibrium. Then, in the large
  population limit, it holds that
  \begin{align}\label{eq:T1}
    \mathbb E[ G_k^{(n)} ] & = \frac{\theta}{k} \frac{(n)_{k \downarrow}}{(n-1+\rho)_{k \downarrow}} \Big(1 +
    \sum\limits_{m=1}^{\infty} \frac{(k)_{m \uparrow}
      \gamma^m}{(n+\rho)_{m \uparrow} m!}\Big)
  \end{align}
  with $ (a)_{b \uparrow} := a (a+1) \cdots (a+b-1)$ and $(a)_{
    b \downarrow} := a (a-1) \cdots (a-b+1)$.
\end{theorem}

\begin{theorem}[More sample statistics]
  \label{T2} Under the same assumptions as in Theorem~\ref{T1},
  \begin{align}
    \label{eq:7}\mathbb E[A^{(n)}] &= \frac{\theta}{\rho} \left( 1 +
      \sum\limits_{m=1}^{\infty} \frac{\gamma^m}{(1+\rho)_{m \uparrow}} \right),\\
    \label{eq:8}\mathbb E[D^{(n)}] &= \frac{\theta}{1 + \rho} \left( 1 +
      \sum\limits_{m=1}^{\infty} \frac{\gamma^m}{(2+\rho)_{m \uparrow} } \right),\\
    \label{eq:9}\mathbb E[G^{(n)}] &= \theta
    \sum_{k=0}^{n-1}\frac{1}{k+\rho} + \theta \sum_{m=1}^\infty
    \frac{\gamma^{m}}{m} \Big(\frac{1}{(\rho)_{m \uparrow}} -
    \frac{1}{(n+\rho)_{m \uparrow}}\Big)
  \end{align}
  in the large population limit.
\end{theorem}

\begin{remark}[Behavior of~\eqref{eq:T1}--\eqref{eq:9}]
  {\em The infinite sums in \eqref{eq:T1} -- \eqref{eq:9} are all
    finite as can be seen by a comparison with the exponential
    series. Note further that \eqref{eq:7} and \eqref{eq:8} do not
    depend on the sample size $n$, while \eqref{eq:9} shows a
    nontrivial dependence on the sample size: The size of the
    accessory genome grows logarithmically with $n$ if $n$ is large
    enough.}
\end{remark}

\begin{theorem}[Second moment of the number of genes]
  \label{T3} Under the same assumptions and in the large population
  limit as in Theorem~\ref{T1}, we have, in the limit $\gamma\to 0$,
  \begin{align}
    \label{eq:T3q}
      \mathbb V[A^{(1)}] & = \frac{\theta}{\rho} \Big( 1 +
      \frac{1}{1+\rho}\gamma + \Big(\frac{1}{(1+\rho)(2+\rho)} \\ &
      \qquad \qquad \qquad + \notag
      \frac{\theta}{(1+\rho)^2(3+2\rho)(2+7\rho+6\rho^2)}\Big)\gamma^2
      \Big) + \mathcal O(\gamma^3),\\
    \label{eq:T3e}
      \mathbb V[D^{(2)}] & = \frac{\theta}{1+\rho} \Big( \frac 12 +
      \frac{\theta}{(1+\rho)(1+2\rho)} + \Big( \frac{1}{2(2+\rho)} \\
      & \quad + \theta \frac{2 (12 + 110 \rho + 248 \rho^2 + 209
        \rho^3 + 60 \rho^4)}{(1 + \rho) (2 + \rho) (1 + 2 \rho)^2 (3 +
        2 \rho) (2 + 3 \rho) (6 + 5 \rho)}\Big) \gamma \Big) +
      \mathcal O(\gamma^2).\notag
  \end{align}
\end{theorem}

\begin{remark}[Higher order terms in~\eqref{eq:T3q}
  and~\eqref{eq:T3e}]
    {\em Although computationally intensive, it should be straightforward
    to improve the approximations in Theorem \ref{T3} for small
    $\gamma$ to higher orders $\mathcal O(\gamma^n)$. In our proof, we
    use an ancestral perspective which includes up to two HGT events,
    leading to the order $\gamma^2$. Including more than two HGT
    events will result in higher order terms, but will lead to an
    increasing amount of genealogies which must be considered. For the
    second order result in \eqref{eq:T3e}, we had to consider more
    than~5000 genealogies.
  }
\end{remark}


\noindent
Before we prove our Theorems, let us give some biological implications
and relations to previous work in the biology literature.

\section{Discussion: Biological Implications}
\label{sec:disc}
Unraveling the amount of HGT shaping bacterial diversity can today be
tackled using a growing amount of genomic data.  In particular,
several datasets from closely related strains, which are of the same
bacterial species are available today
\citep{Medini2005,Tettelin2005,Tettelin2008}.  In such datasets, genes
present in all genomes of a taxon are called \emph{core genes} while
genes present in only some but not all individuals comprise the
\emph{accessory genome}. The latter set of genes is further split into
the medium-frequency \emph{shell} of genes and the \emph{cloud} of
genes of low frequency \citep{Koonin2008}.

~

HGT comes in two flavors, either between or within populations. As for
HGT between populations, a variant of the infinitely many genes model
from \cite{BaumdickerHessPfaffelhuber2010} was introduced by
\cite{Haegeman:2012:BMC-Genomics:22613814}, who couple gene gain and
loss events in order to obtain a genome of constant size. However,
this is in contrast to available data, since flexible genomes of
bacteria usually come with different genome sizes.  An interesting
extension of the infinitely many genes model was studied in
\cite{Collins:2012:Mol-Biol-Evol:22752048}; see also
\cite{Lobovsky2013}. Here, different random trees, including the
coalescent tree, were used as underlying genealogies as well as
different classes of genes, each class with its own rate of gene gain
and loss. It was found that the coalescent produces a good fit with
data, and it is likely that the rate of gene gain and loss depends on
the gene.

In contrast to the vast amount of available data on HGT in bacteria,
mathematical models for HGT within a population are hardly
available. A first approach of the population genomics of bacteria was
made by \cite{Novozhilov:2005:Mol-Biol-Evol:15901840}, extending a
model from \cite{Berg:2002:Mol-Biol-Evol:12446817}. Here, a birth and
death process is used in order to describe the evolution of the
frequency of a single gene under selection under within-population
horizontal transmission (‘‘infection’’), mutation (leading to loss of
the gene) and population size changes. However, this study is limited
since only a single gene is considered, but bacterial genomes are
comprised of several hundreds of genes, each of which may be under
selection and horizontal gene transfer. In
\cite{Mozhayskiy:2012:BMC-Bioinformatics:22759418}, a simulation study
was carried out, taking selective forces into account which arise from
gene regulatory networks, i.e.\ epistasis of presence and absence of
genes. Finally, \cite{VoganHiggs2011} present a macro-evolutionary
model in a constant environment and conclude that HGT was probably
favorable in early evolution since loss of genes is frequent, but
later, when genomes are rather adapted to the environment, HGT is not
favorable and gene losses are rarer.

~

Conceptually, HGT within and between populations are different. Above
all, the \emph{tree of life} has become a classical way of thinking
about inheritance since Darwin's \emph{Origin of Species}. However,
the abundance of HGT within population counteracts the tree-like
structures evolutionary biologists like to think about. Results are
phylogenetic networks, which display at the same time the joint
evolutionary fate of many genes
\citep{Huson:2011:Genome-Biol-Evol:21081312,Dagan:2011:Trends-Microbiol:21820313},
in addition to other reticulate events such as hybridization and
incomplete lineage sorting. It is becoming clear that any genealogical
tree of bacteria which have a flexible genome is at most a tree of 1\%
of all genetic material \citep{Dagan:2006:Genome-Biol:17081279}, which
may eventually lead to a paradigm shift in evolutionary biology of
prokaryotes \citep{Koonin:2012:Front-Cell-Infect-Microbiol:22993722}.

Using the incongruence of genes with the species tree, several
approaches have led to a number of methods to estimate HGT rates and
identify the corresponding genes
\citep{Lawrence20021,Kunin2003,Nakhleh2005,Linz2007,Didelot2010}.
Current estimates show that at least 32\% of the genes in prokaryotic
populations have been horizontally transferred at some point
\citep{Koonin:2001:Annu-Rev-Microbiol:11544372,Dagan:2007:Proc-Natl-Acad-Sci-U-S-A:17213324}. It
may even be argued that this number is still a lower bound because
only a fraction of all events of HGT can be seen in data, either
because the transferred gene is subsequently lost or the pattern is in
accordance to vertical gene transfer
\citep{Gogarten:2002:Mol-Biol-Evol:12446813}.

Note that the distinction of HGT within and between populations points
to the long-standing question of a clear definition of a bacterial
species \citep{Fraser2009}. In our approach, we at least assume that
the entity of a bacterial population exists.

~

Recently, the concepts of \emph{open} and \emph{closed} pangenomes
were introduced \citep{Medini2005}. If, after sequencing a finite
number of genomes, all genes present in the population are found, one
speaks of a \emph{closed} pangenome. If new genes are found even after
sequencing many genomes, the pangenome is called \emph{open}.

In the infinitely many genes model with HGT, one can not identify a
sharp transition between open and closed pangenomes in the limit of
large sample sizes ($n \to \infty$). Theorem \ref{T2}
(see~\eqref{eq:9}) shows that an infinite population possesses an
infinite number of genes, regardless of the parameters $\gamma,\rho$
and $\theta$. However, a closer look reveals that almost all of these
genes are in extreme low frequency. Nonetheless it is possible to give
a quantitative impression how typical rare genes are in a sample in
the presence of HGT. It is not hard to see that abundant HGT (i.e.\ a
high value of $\gamma$) implies that most genes are in
high-frequency. In other words, sequencing a new individual hardly
leads to new genes which were not seen before. This impact of openness
and closeness of the pangenome can as well be seen from
Figure~\ref{fig2}.


\begin{figure}
  \begin{center}
    \includegraphics[width=4in]{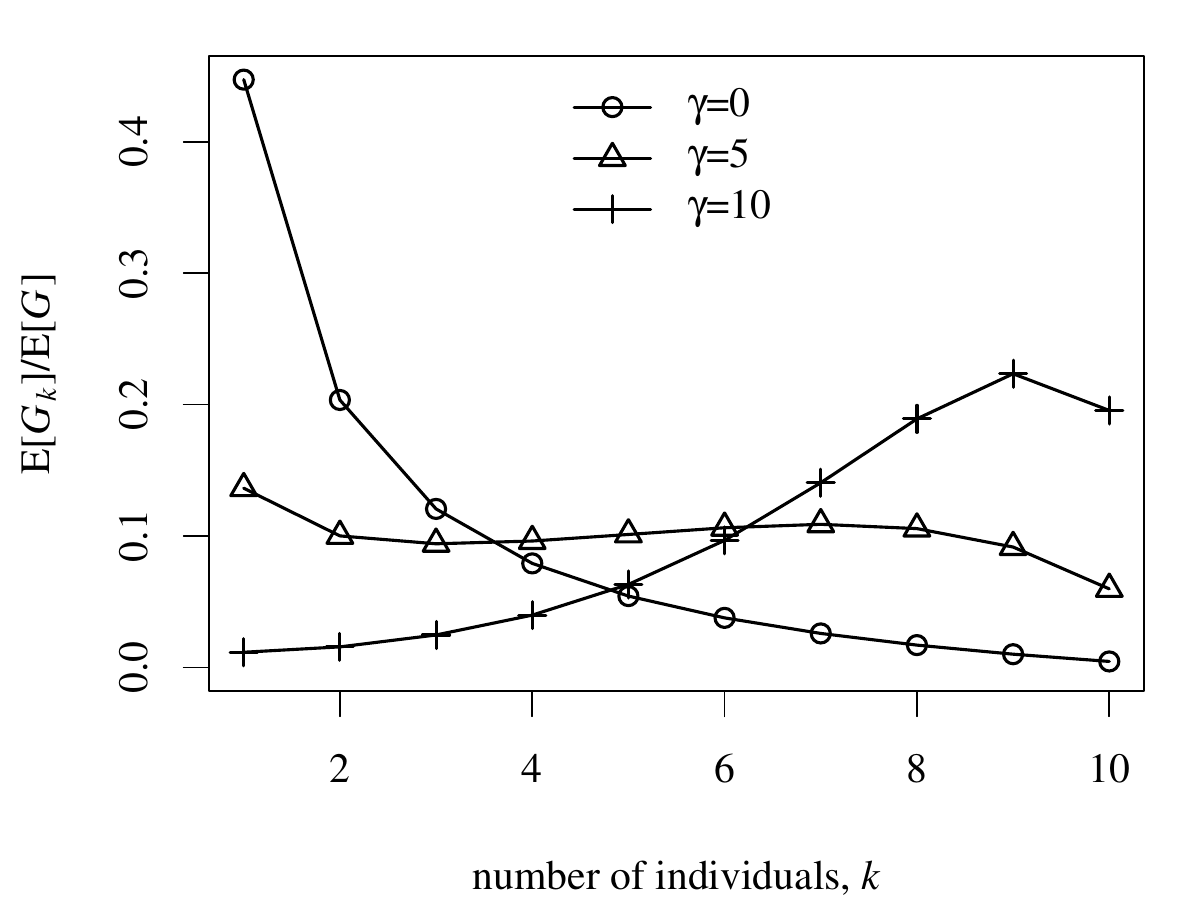}
    \caption{\label{fig2} The expected gene frequency spectrum from
      Theorem~\ref{T1} is highly dependent of $\gamma$, the rate of
      horizontal gene flow. For high values of $\gamma$, most genes
      are in high frequency, leading to a closed pangenome. We use
      $\rho=2$ and sample size $n=10$ in the figure.}
  \end{center}
\end{figure}

~

Although we have made an attempt to include the important evolutionary
force of HGT within a population, the model presented here can still be
extended. We think that the following approaches are conceivable:
\begin{itemize}
\item \emph{Gain, loss and transfer of multiple genes}: The exact
  mechanisms of gene gain, loss and HGT are still under
  study. However, it seems clear that several genes can be gained or
  lost at once.
\item \emph{Gene families}: Frequently, a single gene is present not
  only once but several times in a bacterial genome. The reason can
  either be a copying event along its ancestral line, or the gene is
  introduced by HGT although it was already present.
\item \emph{Gene synteny}: The order of genes in the genome is called
  gene synteny. In our model, the synteny of genes is not modeled, but
  can be observed in genomic data. Above all, gene synteny can give
  hints of events of horizontal gene transfer, since the order of
  genes can be different in donor and recipient.
\item \emph{Mobile genetic elements}: There are parts of the genome
  like mobile elements which are more likely to be transferred
  horizontally. Examples are transposons, plasmids and gene cassettes,
  i.e.\ horizontal gene transfer even within a single cell can be
  considered.
\end{itemize}



\section{The Ancestral Gene Transfer Graph}
\label{S:AGTG}
Since the seminal work of \cite{Kingman1982} and \cite{Hudson1983},
the genealogical view is a powerful tool in the analysis of population
genetic models. Here, we will give a genealogical construction in
order to obtain the distribution of $\mathcal G^N_1,...\mathcal G^N_n$
for a sample of size $n\in\mathbb N$ in the large population limit of
the Moran model with horizontal gene transfer in equilibrium. The
resulting genealogy is denoted the \emph{Ancestral Gene Transfer
  Graph} (AGTG). In this random graph, every ancestral line splits at
constant rate $\gamma/2$ per gene due to a \emph{potential} gene
transfer event. We note that such events leading to potential
ancestors are well-known for the ancestral selection graph (ASG) of
\cite{NeuhauserKrone1997} and \cite{KroneNeuhauser1997}. However,
potential ancestors in the ASG arise by fitness differences within the
population while the potential ancestors within the AGTG may take
effect by events of horizontal gene transfer.

We start with the construction of the genealogy for a single gene and
come to the full picture including all genes afterwards.

\begin{definition}[The AGTG for a single gene]
  \label{def:AGTGsingle}
  Consider a random graph $\mathcal A_n$ which arises as follows:
  Starting with~$n$ lines, denoted $i=1,...,n$,
  \begin{itemize}
  \item each (unordered) pair of lines coalesces at rate~1,
  \item each line disappears at rate $\rho/2$ (meaning that the gene
    was lost),
  \item each line splits in two lines at rate $\gamma/2$ (meaning that
    the gene was horizontally transferred from another individual which was so far no ancestor of the original $n$ cells,
    such that the gene can now have two different origins).
  \end{itemize}
  Sample a single point $E$ uniformly at random according to the
  length measure from the graph. (This point determines the time when
  the gene under consideration was gained.) For every line $1,...,n$,
  let $G_i=1$ if there is a direct (i.e.\ increasing in time) path
  from $i$ to $E$ and $G_i=0$ otherwise. Then, $(G_1,...,G_n)$ is
  denoted the \emph{gene distribution of a single gene read off from
    the AGTG}.
\end{definition}

\begin{SCfigure}[1.6][ht]
  \begin{tikzpicture}
    \node at (1.65,5.2){$\mathcal A_4$};
    \node[scale=1.6] at (0.25,1.9){$\bullet$};
    \node[scale=1.6] at (1.65,4.8){$\bullet$};
    \draw[thick] (0,0) -- (0,1) -- (0.5,1) -- (0.5,0);
    \draw[thick] (1,0) -- (1,2) -- (1.5,2) -- (1.5,0);
    \draw[thick] (1.25,2) -- (1.25,2.7) -- (1.65,2.7) -- (1.65,4.8); 
    \draw[thick] (0.25,1) -- (0.25,1.9);
    \draw[thick] (0.5,0.7) -- (0.9,0.7) arc (180:0:0.1cm) -- (1.4,0.7) arc (180:0:0.1cm) -- (2,0.7) -- (2,2.7) -- (1.25,2.7);
    \draw[thick] (2,1.6) -- (2.5,1.6) -- (2.5,3.2);
    \node[scale=1.6] at (2.5,3.2){$\bullet$};
    \node[scale=1] at (1.65,3.1){$\blacktriangledown$};
    \node at (0,-0.25) {1};
    \node at (0.5,-0.25) {2};
    \node at (1,-0.25) {3};
    \node at (1.5,-0.25) {4};
    \node at (0,-0.6) {\xmark};
    \node at (0.5,-0.6) {\cmark};
    \node at (1,-0.6) {\cmark};
    \node at (1.5,-0.6) {\cmark};
  \end{tikzpicture}
  \caption{\label{fig:1} In the construction of $\mathcal A_4$, the
    AGTG for a single gene, start with $4$ lines at the bottom of the
    figure. Every pair of lines coalesces with rate 1, and every line
    splits at rate $\gamma/2$ and disappears (marked by $\bullet$) at
    rate $\rho/2$. The sampled point $E$ (marked by
    $\blacktriangledown$ for a gene gain event) determines $G_1 = 0$
    (indicated by the \xmark) and $G_2=G_3=G_4=1$ (indicated by the
    \cmark's). }
\end{SCfigure}

\noindent
For later use, we show that all moments of the length of the AGTG are
finite. In particular, the length is almost surely finite, and the
uniform distribution according to the length measure, from which $E$
is picked, is well-defined.

\begin{lemma}[Length of AGTG for a single gene has finite
  moments]
  \label{l:length}
  \mbox{}\\ \sloppy Let $\mathcal A_n$ be the AGTG for a single gene
  from Definition~\ref{def:AGTGsingle} and let $L(\mathcal A_n)$ be
  its length. Then, $\mathbb E[L(\mathcal A_n)^k]<\infty$ for all
  $k=1,2,...$
\end{lemma}

\begin{proof}
  The number of lines in $\mathcal A_n$ is a birth-death process with
  birth rate $\widehat\lambda_i = \gamma i/2$ and death rate
  $\widehat\mu_i = \binom i 2 + \rho i/2$ (when there are~$i$ lines)
  and $0$ as absorbing state. Since the length during times with~$i$
  lines increases at rate $i$, $L(\mathcal A_n)/2$ is distributed as
  the hitting time $T$ of~0 of a birth-death process $(Z_t)_{t\geq 0}$
  with rates $\lambda_i = \gamma$ and $\mu_i = i-1+\rho$, $i=1,2,...$
  and absorbing state $0$. In order to show finite moments of $T$,
  note that the process $(\check Z_t)_{t\geq 0}$ with $\check Z_t =
  Z_t-1$ is bounded from above by a birth-death process with birth
  rates $\check\lambda_k=\gamma$ and death-rate $\check \mu_k =
  k$. In other words, $(\check Z_t)_{t\geq 0}$ is the number of
  customers in an $M/M/\infty$-queue. Let $S$ denote the partial busy
  period of this queue (i.e.\ the first time when the queue is
  empty). Moreover, when $\check Z_t=0$ we have that $Z_t=1$ and there
  is a chance $\rho/(\gamma+\rho)$ that $T$ is reached after an
  $\exp(\gamma+\rho)$-distributed time.  From this construction, we
  see that $T \leq S_1 + \cdots + S_N$ where $S_k \stackrel d = S +
  S'$, and $S'\sim \exp(\gamma+\rho)$ independent from $S$, and $N
  \sim \text{geom}(\rho/(\gamma+\rho))$, all $S_k$'s being
  independent. Hence, the assertion follows from finite moments of $S$
  \citep{Artalejo2001}.
\end{proof}

\noindent
We now come to the desired connection between the Moran model with
horizontal gene transfer and the AGTG.

\begin{lemma}[Gene distribution of Moran model and AGTG coincide]
  \label{l:AGTGsingle}  
  Fix $u\in[0,1]$ and let $\mathcal G_1^N,...,\mathcal G_n^N$ be as in
  Remark~\ref{rem:equi}. Then, for $N\to\infty$, the distribution of
  $\mathcal G_1^N(u),...,\mathcal G_n^N(u)$, conditioned on
  $\bigcup_{i=1}^n \mathcal G_i^N(u) \neq \emptyset$, converges weakly
  to the distribution of $(G_1,...,G_n)$ from
  Definition~\ref{def:AGTGsingle}.
\end{lemma}


\begin{proof}[Proof of Lemma~\ref{l:AGTGsingle}]
  Consider the graphical construction of a Moran model with horizontal
  gene flow from Definition~\ref{def:moran}, run between times
  $-\infty$ and~0.  Let $\mathcal G^N_i(-t)$ be the finite measure
  describing the genome of individual $i$ at time $-t$. If we consider
  only a single gene $u\in[0,1]$, we can use the following procedure
  in order to obtain the genealogy and distribution of gene $u$ in
  $\mathcal G_1^N,...,\mathcal G_n^N$:
  \begin{enumerate}
  \item[(a)] Restrict the Moran model to (i) resampling events, (ii)
    potential gene loss events of gene $u$ at rate $\rho/2$ per line,
    (iii) potential horizontal gene transfer events of gene $u$ at
    rate $\gamma/(2N)$ per pair of individuals.
  \item[(b)] Put gene gain events on all lines according to a Poisson point
    process with intensity $(\theta/2) du$.
  \end{enumerate}
  Clearly, by (a)\, we can determine the coordinates $(i,-t)$ with the
  property, that $u\in \bigcup_{j=1}^n \mathcal G_j^N(0)$ iff
  $u\in\mathcal G_i^N(-t)$. This subgraph is a random graph, which can
  be constructed from time 0 backwards as follows: Starting with $n$
  lines, any pair of lines coalesces at rate~1, every line is killed
  at rate $\rho/2$ and every line splits in two lines (due to a
  horizontal gene transfer event) at rate $\gamma (N-k)/(2N)$ if there
  are currently $k$ lines within the graph. (For the latter rate,
  observe that the donor of gene $u$ might already be part of the
  graph.) Hence, as $N\to\infty$, this random graph converges (weakly)
  to the AGTG for a single gene as in
  Definition~\ref{def:AGTGsingle}. In addition, for small
  $\varepsilon$, genes in $(u-\varepsilon, u+\varepsilon)$ are gained
  at most once on this graph. Hence, when conditioning on the event of
  a gene gain of gene $u$ on the random graph (i.e.\ the Poisson point
  process has a point $(x,u)$), by well-known properties of Poisson
  processes, this event is uniformly distributed on the graph. In
  other words, the distribution is the same as that of $E$ in
  Definition~\ref{def:AGTGsingle}.
\end{proof}

~

\noindent
While the construction of the genealogy of a single gene was
straight-forward, considering all genealogies of all genes seems to be
harder. The reason is that there can be infinitely many genes, and
each of these genes comes with its own events of gene gain, loss and
horizontal gene flow. Even worse, we can decide on the presence or
absence of a gene only if we know if there was a gene gain event
somewhere along the genealogy, which means that we have to follow all
(uncountably many) potential genes back in time.

We will resolve such difficulties by constructing (countably) many
potential genealogies, which all rely on the same clonal genealogy of 
$n$ bacterial cells,
and model gene gain events along them. The
result is the ancestral gene transfer graph (AGTG) for infinitely many
genes. An illustration is found in Figure~\ref{fig:Ls}.

\begin{figure}

  \begin{tikzpicture}
    \draw[thick, <- ] (-0.5,1.5) --(-0.5,3.5);
    \node[above] at (-0.5,3.5){time};
    \node at (0.7,6.7){$\mathcal A_4^{(0)}$};
    \draw[thick] (0,0) -- (0,1) -- (0.5,1) -- (0.5,0);
    \draw[thick] (1,0) -- (1,2) -- (1.5,2) -- (1.5,0);
    \draw[thick] (0.25,1) -- (0.25,4) -- (1.25,4) -- (1.25,2);
    \draw[thick] (0.75,4)--(0.75,6);
  \end{tikzpicture}
  \hfill \hfill
  \begin{tikzpicture}
    \node at (0.7,6.7){$\mathcal A_4^{(1)}$};
    \draw[thick,dashed] (0.25,2) -- (0.25,4) -- (0.75,4);
    \draw[thick,dashed] (0.75,4)--(0.75,6);
    \node at (0.25,1.9){$\bullet$};
    \node at (0.4,1.9){$1$};
    \node at (.75,4.8){$\bullet$};
    \node at (.9,4.8){$1$};
    \draw[thick] (0,0) -- (0,1) -- (0.5,1) -- (0.5,0);
    \draw[thick] (1,0) -- (1,2) -- (1.5,2) -- (1.5,0);
    \draw[thick] (1.25,2) -- (1.25,4) -- (0.75,4) -- (0.75,4.8);
    \draw[thick] (0.25,1) -- (0.25,1.9);
    \draw[thick] (0.5,0.7) -- (0.9,0.7) arc (180:0:0.1cm) -- (1.4,0.7) arc (180:0:0.1cm) -- (2,0.7) -- (2,2.7) -- (1.25,2.7);
    \node at (.4,.7){$1$};
    \draw[thick] (2,1.6) -- (2.5,1.6) -- (2.5,3.2);
    \node at (1.85, 1.6){$1$};
    \node at (2.5,3.2){$\bullet$};
    \node at (2.65,3.2){$1$};
  \end{tikzpicture}
   \hfill
  \begin{tikzpicture}
    \node at (0.7,6.7){$\mathcal A_4^{(2)}$};
    \draw[thick,dashed] (0.25,1) -- (0.5,1) -- (0.5,0.5);
    \draw[thick,dashed] (0.75,4.2)--(0.75,5);
    \draw[thick,dashed] (0.5,0.7) -- (0.9,0.7) arc (180:0:0.1cm) -- (1.4,0.7) arc (180:0:0.1cm) -- (2,0.7) -- (2,2.7) -- (1.25,2.7);
    \draw[thick,dashed] (2,1.6) -- (2.5,1.6) -- (2.5,2.5) -- (2.75,2.5);
    \draw[thick] (1.25,2) -- (1.25,4) -- (0.75,4);
    \draw[thick] (0,0) -- (0,1) -- (0.25,1) -- (0.25,4) -- (0.75,4) -- (0.75,4.2);
    \node at (.75,5.8){$\bullet$};
    \node at (0.9,5.8){$2$};
    \draw[thick] (0.5,0) -- (0.5,0.5);
    \node at (.5,0.5){$\bullet$};
    \node at (0.35,0.5){$2$};
    \draw[thick] (1,0) -- (1,2) -- (1.5,2) -- (1.5,0);
    \draw[thick] (1.5,0.95) -- (1.9,0.95) arc (180:0:0.1cm) -- (3,0.95);
    \node at (1.35,0.95){$2$};
    \draw[thick] (3,0.95) -- (3,2.5) -- (2.75,2.5) -- (2.75,5) -- (.75,5) -- (.75,5.8);
    \node at (.75,4.2){$\bullet$};
    \node at (.9,4.2){$2$};
  \end{tikzpicture}
  \caption{\label{fig:Ls}In the construction of the AGTG, start with
    the clonal genealogy for the bacterial cells of the sample (here of size~4), i.e.\ with
    $\mathcal A_4^{(0)}$. Then, in order to obtain the genealogy of
    the first potential gene, construct $\mathcal A_4^{(1)}$ by additional
    splitting events (at rate $\gamma/2$), loss events (at rate
    $\rho/2$), both marked by $1$, and coalescence
    events. Iteratively, construct $\mathcal A_4^{(n+1)}$ by keeping
    all lines in $\cup_{i=0}^n \mathcal A_4^{(i)}$ and adding
    splitting, loss and coalescence events. In the three figures, the
    vertical solid lines are the ones where -- potentially -- the corresponding
    gene can be gained. This means that gene~2 is present only if the
    second gain event at time $T_2$ occurs at a time which is smaller
    than the sum of all vertical lines in $\mathcal A_4^{(2)}$,
    i.e.\ smaller than $L(\mathcal A_4^{(2)})$. If it occurs, it is
    put uniformly on the solid vertical lines.  }
\end{figure}
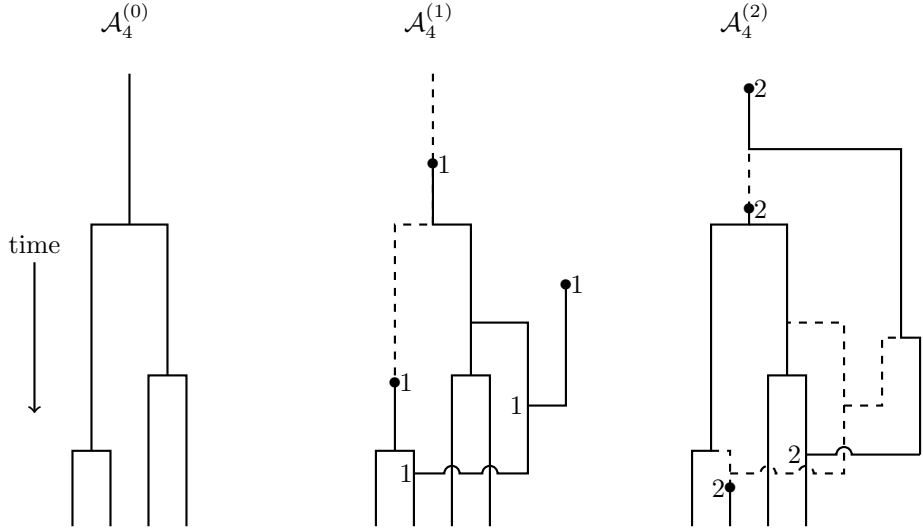

\begin{definition}[The AGTG for infinitely many genes]
  \label{def:ATGTmany}
  \sloppy Consider a sequence $\mathcal A_n^{(0)}, \mathcal A_n^{(1)},
  \mathcal A_n^{(2)}, ...$ of coupled random graphs which arise as
  follows:
  \begin{itemize}
  \item $\mathcal A_n^{(0)}$ is distributed according to Kingman's
    coalescent, i.e.\, starting with $n$ lines, each (unordered) pair
    of lines coalesces at rate~1 (and the graph is stopped as soon as
    the number of blocks (lines) is one).
  \end{itemize}
  Given $\mathcal A_n^{(0)}$, the random graph $\mathcal A_n^{(1)}$
  gives all potential ancestors of gene~1 and is constructed such
  that: Starting with the same $n$ lines as in $\mathcal A_n^{(0)}$,
  \begin{itemize}
  \item each line splits in a continuing and an incoming line at rate
    $\gamma/2$ (meaning that the gene was horizontally transferred
    from the incoming line). If the line was part of $\mathcal
    A_n^{(0)}$, the continuing line runs along $\mathcal A_n^{(0)}$ as
    well. The resulting splitting event is marked with \enquote{1}
  \item each line is terminated by a loss event, also marked with
    \enquote{1}, at rate $\rho/2$ (indicating that gene 1 was lost).
  \item in addition to coalescences of lines within $\mathcal
    A_n^{(0)}$, each (unordered) pair of lines in $(\mathcal
    A_n^{(1)}\setminus \mathcal A_n^{(0)})^2$ and in $\mathcal
    A_n^{(0)} \times (\mathcal A_n^{(1)}\setminus \mathcal A_n^{(0)})$
    coalesces at rate 1 (and $\mathcal A_n^{(1)}$ is stopped as soon
    as the number of lines is zero).
  \end{itemize}
  Given $\mathcal A_n^{(0)},...,\mathcal A_n^{(k)}$, the random graph
  $\mathcal A_n^{(k+1)}$ gives all potential ancestors of gene~$k+1$
  and is constructed such that: Starting with the same $n$ lines as in
  $\mathcal A_n^{(0)}$,
  \begin{itemize}
  \item each line splits in a continuing and an incoming line at rate
    $\gamma/2$ (meaning that the gene was horizontally transferred
    from the incoming line). If the line was part of
    $\bigcup_{j=0}^k\mathcal A_n^{(j)}$, the continuing line runs
    along $\bigcup_{j=0}^k\mathcal A_n^{(j)}$ as well. The resulting
    splitting event is marked with \enquote{$k+1\!$}.
  \item each line is terminated by a loss event, also marked with
    \enquote{$k+1\!$}, at rate $\rho/2$ (indicating that gene $k+1$ was
    lost).
  \item in addition to coalescences of lines within $\bigcup_{j=0}^k
    \mathcal A_n^{(j)}$, each (unordered) pair of lines in $(\mathcal
    A_n^{(k+1)}\setminus \bigcup_{j=0}^k \mathcal A_n^{(j)})^2$ and in
    $\bigcup_{j=0}^k \mathcal A_n^{(j)} \times (\mathcal
    A_n^{(k+1)}\setminus \bigcup_{j=0}^k\mathcal A_n^{(j)})$ coalesces
    at rate 1 (and the graph $\mathcal A_n^{(k+1)}$ is stopped as soon
    as the number of lines is zero).
  \end{itemize}
  After the construction of all graphs we consider for each $k$ only the relevant parts of $\mathcal A_n^{(k)}$, 
  i.e.\ those parts that
  can be
  reached from at least one of the leaves by running through
  coalescence events or splitting events marked with $k$.
  In Figure~\ref{fig:Ls} these parts are shown as solid lines,
  while the additional lines, necessary for the construction but unreachable
  from the leaves, are shown as dashed lines.

  In order to model gene gain events, consider the events
  $(T_m,U_m)_{m=1,2,...}$ of a Poisson point process on
  $[0,\infty)\times [0,1]$ with intensity measure $\tfrac 12\theta
  \,dt \, du$ (ordered by their first coordinate). For all $k$,
  \begin{itemize}
  \item let $L(\mathcal A_n^{(k)})$ be the length of $\mathcal
    A_n^{(k)}$, i.e.\ the length of all vertical solid lines in Figure~\ref{fig:Ls}.
    If $T_k\leq L(\mathcal A_n^{(k)})$, pick a point $E_k$
    uniformly at random according to the length measure on $\mathcal
    A_n^{(k)}$. (This point determines the time and line when the gene
    $U_k$ was gained.)
  \end{itemize}
  Finally, for every $i=1,...,n$, let $U_k\in{\mathcal G_i}$ if there
  is a direct (i.e.\ increasing in time) path from $i$ to $E_k$ in $\mathcal A^{(k)}_n$. Then,
  $({\mathcal G_1},..., {\mathcal G_n})$ is denoted the \emph{Gene
    distribution read off from the AGTG in the infinitely many genes
    model}.
\end{definition}

\begin{remark}[Alternative way of distributing gain events on the
  AGTG]
  {\em In the last step of constructing the AGTG, we used the
    condition $T_k\leq L(\mathcal A_n^{(k)})$ in order to distribute a
    uniformly chosen point $E_k$ on $\mathcal A_n^{(k)}$. In
    distribution, the same result is achieved as follows: If $T_k\leq
    L(\mathcal A_n^{(k)})$, choose a way of running through $\mathcal
    A_n^{(k)}$ along all paths at constant speed. Such a path might
    well go back and forth and jump in time. Then, the gene gain event
    is placed after running length $T_k$.}
\end{remark}

\noindent
The following Lemma is the key element in the proofs given in Section
\ref{S:proofs}.

\begin{lemma}[\sloppy Gene distribution from Moran model and AGTG
  coincide]
  \label{l:identity}
  \sloppy Fix $n\in\mathbb N$, let $(\mathcal G_1^{N},...,\mathcal
  G_n^{N})$ be as in Definition~\ref{def:moran} and
  Remark~\ref{rem:equi}, and $({\mathcal G_1},..., {\mathcal G_n})$ as
  in Definition~\ref{def:ATGTmany}. Then,
  \begin{align}
    \label{eq:convRM}
    (\mathcal G_1^{N},...,\mathcal G_n^{N}) \xRightarrow{N\to\infty}
    ( {\mathcal G_1},..., {\mathcal G_n})
  \end{align}
  as well as
  \begin{align}
    \label{eq:convRM2}
    (\mathcal G_{1}^{N} \otimes \cdots \otimes \mathcal G_n^{N})
    \xRightarrow{N\to\infty} (\mathcal G_i \otimes \cdots \otimes
    \mathcal G_n)
  \end{align}
\end{lemma}

\begin{remark}[Interpretation of~\eqref{eq:convRM}
  and~\eqref{eq:convRM2} and a convergence criterion]\mbox{}
  {\em \label{rem:kall}
    \begin{enumerate}
    \item Note that the space of finite (counting) measures on $[0,1]$
      is equipped with the topology of weak (or vague)
      convergence. (In our proofs we will use Skorohod's Theorem which
      states that weak convergence is equivalent to almost sure
      convergence on an appropriate probability space;
      cf.~\citealp{Kallenberg2002}, Theorem~4.30.) In addition, we
      will interpret a vector $(\xi_1,...,\xi_n)$ of counting measures
      on $[0,1]$ as a counting measure on $\{1,...,n\}\times
      [0,1]$. Henceforth, we write $\mathcal G^N (\cup_{i=1}^n
      \{i\}\times A_i) = \prod_{i=1}^n \mathcal G^N_i (A_i)$ for
      $A_1,...,A_n\in\mathcal B([0,1])$ such that~\eqref{eq:convRM} is
      the same as
      $$ (\langle \mathcal G_1^{N}, f_1\rangle,...,\langle \mathcal
      G_n^{N}, f_n\rangle) \xRightarrow{N\to\infty} ( \langle
      {\mathcal G_1}, f_1\rangle,..., \langle{\mathcal G_n}, f_n,
      \rangle)$$ for all $f_1,...,f_n\in\mathcal C([0,1])$.
    
      Since $1\in\mathcal C([0,1])$,~\eqref{eq:convRM} also implies
      the convergence of total masses of $\mathcal G_i^{N}$. In
      addition,~\eqref{eq:convRM2} is stronger because total masses of
      products, $\langle \mathcal G_1^{N}, 1\rangle\cdots \langle
      \mathcal G_n^{N}, 1\rangle$ converge as well.
    \item In our proof, we use the following convergence criterion from
      \cite{Kallenberg2002}, Proposition 16.17, here adapted for random
      measures on a compact space:
      \begin{quote}
        Let $\xi, \xi^1, \xi^2,...$ be random counting measures on a
        compact metric space $I$, where $\xi$ is simple. Then,
        $\xi_n\xRightarrow{N\to\infty}\xi$, if (i) $\mathbb
        P(\xi_n(A)=0) \xrightarrow{n\to\infty} \mathbb P(\xi(A)=0)$
        for all open $A\subseteq I$ and (ii) $\limsup_{n\to\infty}
        \mathbb E[\xi_n(A)] \leq \mathbb E[\xi(A)] < \infty$ for all
        compact $A\subseteq I$.
      \end{quote}
    \end{enumerate}
    }
\end{remark}

\begin{proof}[Proof of Lemma~\ref{l:identity}]
  We proceed in five steps. In Step~1, we define another set of models
  for a population of size $N$ with horizontal gene transfer, indexed
  by $K$, in which $I=[0,1]$ is separated into $K$ classes of genes,
  $\Delta_i^K:=[(i-1)/K; i/K), i=1,...,K$. For the resulting genomes,
  denoted $(\mathcal G_1^{N,K},..., \mathcal G_n^{N,K})$, we show in
  Step~2 that the genealogies of $(\mathcal G_1^{N,K},..., \mathcal
  G_n^{N,K})$ are given by an AGTG with $K+1$ coupled random
  graphs. The construction of these random graphs can be re-ordered
  such that the limit $K\to\infty$ can be taken easily; see Step~3. In
  Step~4, we let $N\to\infty$ and show the convergence of the coupled
  random graphs to $(\mathcal A_n^{(0)}, \mathcal A_n^{(1)}, ...)$,
  implying the convergence to $(\mathcal G_1,..., \mathcal G_n)$. In
  the last step we show convergence of second moments.

  ~

  \noindent
  {\it Step 1: Definition of $\mathcal G_i^{N,K}$}: Fix $K\in\mathbb
  N$ and set $\Delta_i^K := [(i-1)/K; i/K), i=1,...,K$. We define
  another Moran model (called $\text{Moran}_\Delta$-model) with
  horizontal gene transfer. Briefly, in this model, all genes
  $u\in\Delta_i^K$ follow the same gene loss and gene transfer
  events. Precisely, in addition to resampling events (at rate~1 for
  every ordered pair of individuals), the following events occur:
  \begin{enumerate}
  \item \emph{Gene loss}: For all $k=1,...,K$, a gene loss event
    occurs at rate $\rho/2$ per individual. Upon such an event in
    individual $i$, we have $\mathcal G_i^{N,K}(t) = \mathcal
    G_i^{N,K}(t-)\setminus \Delta_k^K$, i.e.\ all
    genes~$u\in\Delta_k^K$ are lost from~$\mathcal G_i^{N,K}(t)$.
  \item \emph{Gene gain}: For every individual $i$, at rate
    $\theta/2$, choose $U$ uniformly in $[0,1]$. If $U\in\Delta_k^K$,
    set $\mathcal G_i^{N,K}(t) = (\mathcal G_i^{N,K}(t-)\setminus
    \Delta_k^K)\cup \{U\}$, i.e.\ $U$ is the only gene in $\mathcal
    G_i^{N,K}(t)\cap \Delta_k^K$.
  \item \emph{Horizontal gene transfer}: For every (ordered) pair of
    individuals $(i,j)$ and $k=1,...,K$, a horizontal gene transfer
    event occurs at rate $\gamma/(2N)$. For such an event, set $\mathcal
    G_j^{N,K}(t) = \mathcal G_j^{N,K}(t-)\cup (\mathcal
    G_i^{N,K}(t-)\cap\Delta_k^K)$, i.e.\ individual~$i$ is the donor
    of all genes~$u\in\mathcal G_i^{N,K}(t-)\cap\Delta_k^K$ to the
    recipient~$j$.
  \end{enumerate}
  Again, $(\mathcal G_1^{N,K}(t),...,\mathcal G_N^{N,K}(t))_{t\geq 0}$
  has a unique ergodic equilibrium; compare with
  Lemma~\ref{rem:equi}. We start it at time $-\infty$ and thus obtain
  the equilibrium measures $(\mathcal G_1^{N,K}:=\mathcal
  G_1^{N,K}(0),...,\mathcal G_N^{N,K}:=\mathcal G_N^{N,K}(0))$ by time
  0.

  ~

  \noindent
  {\it Step 2: $(\mathcal G_1^{N,K},...,\mathcal G_n^{N,K})$ can be
    constructed using $K+1$ random graphs}: Recall the construction of
  the AGTG for a single gene from Definition~\ref{def:AGTGsingle}. We
  extend this construction in order to obtain the distribution of
  $(\mathcal G_1^{N,K},...,\mathcal G_n^{N,K})$. Since $K$ is finite,
  we can proceed by a two-step procedure similar to the proof of
  Lemma~\ref{l:AGTGsingle} in the $\text{Moran}_\Delta$ model. Here,
  we first generate resampling, gene loss and transfer events and subsequently
  introduce gene gain events. So, first consider a Moran
  model with (i) resampling events, (ii) potential gene loss events
  for genes in $\Delta_k^K$ with rate $\rho/2$ along all lines, where
  a transition from $\mathcal G_i^{N,K}(t)$ to $\mathcal
  G_i^{N,K}(t)\setminus \Delta_k^K$ occurs, $k=1,...,K$ and (iii)
  potential gene transfer events of genes in $\Delta_k^K$ with rate
  $\gamma/(2N)$ per pair of individuals, as in $\text{Moran}_\Delta$,
  $k=1,...,K$. Next, introduce gene gain events for all lines and all
  $\Delta_k^K, k=1,...,K$ at rate $\theta/(2K)$, where each new gene is
  assigned a uniformly distributed random variable on $\Delta_k^K$.

  Equivalently, as for the AGTG for a single gene, we can start from
  time $0$ backwards and construct $K+1$ random graphs such that graph
  $k$ describes the possible ancestry of genes in $\Delta_k^K,
  k=1,...,K$. Precisely, start with graph~$0$, which is a coalescent
  started with $n$ individuals (without gene loss and horizontal gene
  transfer events). In graph~1, add gene loss events, valid for all
  $u\in\Delta_1^K$ and gene transfer events, which lead to splits of
  lines in the graph at rate $\gamma (N-m)/(2N)$, if it currently has
  $m$ lines. In addition, at rate $\gamma m/(2N)$, the split of a line
  leads to ancestry to a line which is already within the
  graph. Iteratively, in graph $k$, additional loss and split events,
  valid for genes $u\in\Delta_k^K$, occur. Again, a split might generate
  a line which was already present in graph $0,...,k-1$, and
  otherwise gives a new line. These graphs are denoted $\mathcal
  A_{n,N,K}^{(0)},...,\mathcal A_{n,N,K}^{(K)}$.

  After having constructed all $K+1$ random graphs, graphs $1,...,K$
  are hit by gene gain events, each with rate $\theta /(2K)$.  As
  above, each new gene in graph~$k$ is assigned a uniformly
  distributed random variable on $\Delta_k^K$. In each $\Delta_k^K$,
  keep only the gene which is closest to time $0$, since in
  $\text{Moran}_\Delta$, a new gene in $\Delta_k^K$ overwrites present
  ones. By this procedure, we can read off $(\mathcal
  G_1^{N,K},...,\mathcal G_n^{N,K})$ from the random graphs, which
  are marked by gene gain events. Note that there is at most one gene in each $\Delta_k^K$ for any individual $i$ (i.e.\ $\mathcal
  G_i^{N,K}(\Delta_k^K)\leq 1$) and we claim that graphs $1,...,K$ are
  exchangeable by construction. Indeed, the gene losses and splits of
  $\mathcal A_{n,N,K}^{(j)}$ are only valid for genes in $\Delta_j^K$.
  Hence the crucial part to understand exchangeability is the time a
  line produced by a split within $\Delta_j^K$ needs to merge back to
  the graph $\mathcal A_{n,N,K}^{(0)},...,\mathcal A_{n,N,K}^{(i)}$
  for $i < j$.  The newly generated line merges with each line in
  $\mathcal A_{n,N,K}^{(0)},...,\mathcal A_{n,N,K}^{(i)}$ at rate 1.
  This rate does not depend on whether the new line merges previously
  to a line in $\mathcal A_{n,N,K}^{(i+1)},...,\mathcal
  A_{n,N,K}^{(j-1)} \setminus \mathcal A_{n,N,K}^{(0)},...,\mathcal
  A_{n,N,K}^{(i)}$ or not, as each line in $\mathcal
  A_{n,N,K}^{(i+1)},...,\mathcal A_{n,N,K}^{(j-1)} \setminus \mathcal
  A_{n,N,K}^{(0)},...,\mathcal A_{n,N,K}^{(i)}$ merges as well at rate
  1 with each line in $\mathcal A_{n,N,K}^{(0)},...,\mathcal
  A_{n,N,K}^{(i)}$.  Thus the times to merge to $\mathcal
  A_{n,N,K}^{(0)},...,\mathcal A_{n,N,K}^{(i)}$ are equal in law for
  the line produced by a split in $\Delta_j^K$ and the line in
  $\mathcal A_{n,N,K}^{(i+1)}$ produced at the same time by a split in
  $\Delta_{i+1}^K$.

  ~

  \noindent
  {\it Step 3: $(\mathcal G_1^{N,K},...,\mathcal
    G_n^{N,K})\xRightarrow{K\to\infty} (\mathcal G_1^{N},...,\mathcal
    G_n^{N})$ and the limit can be constructed using countably many
    random graphs}: In the construction of the last step, we reverse
  the order of generating gene gain events and the random
  graphs. First, let $(T_m, U_m)_{m=1,2,...}$ be the points in a
  Poisson point process $\mathcal T$ on $[0,\infty)\times[0,1]$ with
  intensity $\tfrac \theta 2 dt du$, ordered by their first
  coordinates. Instead of constructing the random graphs $0,...,K+1$
  in the order of the intervals $\Delta_1^N,...,\Delta_K^N$ in
  $[0,1]$, we can as well construct the random graphs in the order of
  appearance of gene gain events. Formally, let $K_m:=k$ if $U_m\in
  \Delta_k^K$, i.e.\ $K_m$ gives the number of the interval
  $\Delta_k^K$ in which the $m$th gene gain event
  $(T_m,U_m)_{m=1,2,...}$ appears. Then, let $i_1:=1$ and $i_{r+1} :=
  \inf\{m > i_r: K_m \notin\{K_1,...,K_{m-1}\}$ for $r<K$. This means
  that $K_{i_1},...,K_{i_K}$ is the number of intervals in the order
  of the appearance of the first gene gain within each interval. Most
  importantly, $(\mathcal A_{n,N,K}^{(0)},\mathcal
  A_{n,N,K}^{(1)}...,\mathcal A_{n,N,K}^{(K)})\stackrel d = (\mathcal
  A_{n,N,K}^{(0)}, \mathcal A_{n,N,K}^{(K_{i_1})},...,\mathcal
  A_{n,N,K}^{(K_{i_K})})$ since the Poisson point process $\mathcal T$
  is independent of $(\mathcal A_{n,N,K}^{(1)},...,\mathcal
  A_{n,N,K}^{(K)})$ and the random graphs $(\mathcal
  A_{n,N,K}^{(1)},...,\mathcal A_{n,N,K}^{(K)})$ are exchangeable.

  Now, consider gene gain events on $\mathcal A_{n,N,K}^{(K_1)}$. By
  construction, the first gene gain event at time $T_1$ falls into
  $\Delta_{K_1}^K$. Hence, this graph is hit after an exponentially
  distributed time with rate $\theta/2$. (Note the difference to the
  rate $\theta/(2K)$ from the last step.) In order to model this, take
  $T_1$ (the time of the first gene gain in $\mathcal T$), determine a
  set of paths how to move through $\mathcal A_{n,N,K}^{(K_1)}$ and
  place a gene gain event after time $T_1$. In case the length of
  $\mathcal A_{n,N,K}^{(K_1)}$ is smaller than $T_1$, do
  nothing. Continuing, by construction, $\mathcal
  A_{n,N,K}^{(K_{i_2})}$ (recall $K_{i_2}=K_2$ if $K_2\neq K_1$) is
  hit by a gene gain event, this event occurs by time $T_2$ of
  $\mathcal T$. Again, determine a set of paths how to move through
  $\mathcal A_{n,N,K}^{(K_2)}$ and place a gene gain event after time
  $T_2$, if possible. Continue until $T_{i_K}$ and $\mathcal
  A_{n,N,K}^{(K_K)}$.

  In this construction, we can now let $K\to\infty$, which means that
  we construct infinitely many random graphs, $\mathcal A_{n,N}^{(0)},
  \mathcal A_{n,N}^{(1)},...$ such that the first $K+1$ are
  distributed according to $(\mathcal A_{n,N,K}^{(0)}, \mathcal
  A_{n,N,K}^{(K_{i_1})},...,\mathcal A_{n,N,K}^{(K_{i_K})})$. On these
  \sloppy infinitely many random graphs $\mathcal A_{n,N}^{(0)},
  \mathcal A_{n,N}^{(1)},...$, we can now use all points $(T_m, U_m)$
  in order to construct the resulting genomes, which we denote by
  $\widetilde{\mathcal G}^N$.

  We now show that ${\mathcal G}^{N,K}\xRightarrow{K\to\infty}
  \widetilde{\mathcal G}^N$ as well as ${\mathcal
    G}^{N,K}\xRightarrow{K\to\infty} {\mathcal G}^N$ (the latter being
  the set of genomes from the Moran model), implying that ${\mathcal
    G}^N \stackrel d = \widetilde{\mathcal G}^N$, i.e.\ the genomes
  $\mathcal G^N$ can be constructed from infinitely many random graphs
  by using all points in $\mathcal T$. We use the criterion from
  Remark~\ref{rem:kall}. For the convergence to $\widetilde{\mathcal
    G}^N$, note that both ${\mathcal G}^{N,K}$ and
  $\widetilde{\mathcal G}^N$ can be constructed on a joint probability
  space, using the same (infinitely many) random graphs. The
  difference in construction is that for $\widetilde{\mathcal G}^N$,
  all points in $\mathcal T$ are used, while in ${\mathcal G}^{N,K}$
  only the first points within each $\Delta_i^K$ are used. Moreover,
  as long as at most one gene gain event hits $\mathcal
  A_{n,N,K}^{(i)}$ the random measures $\mathcal G^{N,K}$ and
  $\widetilde{\mathcal G}^N$ agree on $\Delta_i^K$. Hence, we write,
  for any Borel set $A\subseteq \{1,...,n\}\times [0,1]$ and
  $k=0,1,2,...$ and using Lemma~\ref{l:length} in the last step
  \begin{equation}
    \label{eq:conv1}
    \begin{aligned}
      |\mathbb P(\widetilde{\mathcal G}^N(A)=k) & - \mathbb P(\mathcal
      G^{N,K}(A)=k)| \\ & \leq \mathbb P\Big(\bigcup_{i=1}^K
      \mathcal A^{(i)}_{N,K} \text{ hit by 2 gene gain events}\Big) \\
      & \leq K\cdot \mathbb P(\mathcal A^{(1)}_{N,K} \text{ hit by 2
        gene gain events}) \\ & \leq K\cdot \mathbb E[1-\exp(-\theta
      L(\mathcal A_N^{(1)})/(2K))(1+\theta L(\mathcal A_N^{(1)})/(2K))] \\
      & \leq \frac{K \theta^2}{4K^2}\mathbb E[(L(\mathcal A_N^{(1)}))^2]
      \xrightarrow{K\to\infty} 0,
    \end{aligned}
  \end{equation}
  implying (i) of Remark~\ref{rem:kall}.2. Now, let $L(\mathcal
  A_{1,N,K}^{(1)})$ and $L(\mathcal A_{1,N}^{(1)})$ be the lengths of
  the random graphs $\mathcal A_{1,N,K}^{(1)}$ and $\mathcal
  A_{1,N}^{(1)}$, which correspond to $\mathcal G_1^{N,K}$ and $\widetilde{\mathcal G}_1^N$, respectively. Note that $\mathcal A_{1,N,K}^{(1)}
  \stackrel d = \mathcal A_{1,N}^{(1)}$. By construction, we write
  \begin{equation}
    \label{eq:GNK1}
    \begin{aligned}
      \mathcal G^{N,K}_1([0,1]) & = \sum_{m=1}^K 1_{L(\mathcal
        A_{1,N,K}^{(K_m)})\geq T_m},\\
      \widetilde{\mathcal G}^{N}_1([0,1]) & = \sum_{m=1}^\infty 1_{L(\mathcal
        A_{1,N}^{(m)})\geq T_m}.
    \end{aligned}
  \end{equation}
  \sloppy Then, by exchangeability, for a rate-$1$-exponentially
  distributed random variable $X$ and $\mathbb E[L(\mathcal
  A_{1,N,K}^{(1)})] = \mathbb E[L(\mathcal A_{1,N}^{(1)})]\leq \mathbb
  E[L(\mathcal A_{1})]<\infty$ by Lemma~\ref{l:length},
  \begin{equation}
    \label{eq:conv2}
    \begin{aligned}
      \mathbb E[\mathcal G^{N,K}( & \{1,...,n\}\times [0,1])] = n
      \cdot \mathbb E[\mathcal G_1^{N,K}([0,1])] \\ & = nK \cdot
      \mathbb P(L(\mathcal A_{1,N,K}^{(1)}) \geq \tfrac{2K}{\theta} X)
      \\ & = nK \cdot \mathbb E[1-\exp(-\theta L(\mathcal
      A_{1,N,K}^{(1)})/(2K))] \\ & \xrightarrow{K\to\infty} n\theta
      \cdot \mathbb E[ L(\mathcal
      A_{1,N}^{(1)})] = n \cdot \mathbb E[\widetilde{\mathcal G}_1^{N}([0,1])] \\
      & = \mathbb E[\widetilde{\mathcal G}^{N}( \{1,...,n\}\times [0,1])],
    \end{aligned}
  \end{equation}
  which gives (ii) of the convergence criterion given in
  Remark~\ref{rem:kall}.2. (Note that the finiteness of the right hand
  side of the last equation can be seen from $\mathbb E[L(\mathcal
  A_{1,N}^{(1)})]<\infty$; see Lemma~\ref{l:length}.) Next, we come to
  the convergence ${\mathcal G}^{N,K}\xRightarrow{K\to\infty}
  {\mathcal G}^N$. Again, we observe that both random measures can be
  constructed on one probability space. Here, use the $K+1$ random
  graphs in order to construct ${\mathcal G}^{N,K}$ first and draw
  them as a part of a graphical construction of the
  $\text{Moran}_\Delta$-model, starting at time~$0$. Note that in the
  $\text{Moran}_\Delta$-model, gene gain events for a gene
  $u\in\Delta_k^K$ can lead to loss of another gene $v\in\Delta_k^K$,
  if the line of the gene gain event carries gene $v$. For such genes,
  which are lost in the $\text{Moran}_\Delta$-model, put additional
  gene loss and transfer events in the (regular) Moran model. Again,
  we claim that ${\mathcal G}^{N,K} = {\mathcal G}^N$ if every random
  graph $\mathcal A_{n,N,K}^{(1)},...,\mathcal A_{n,N,K}^{(K)}$ is hit
  by at most one gene gain event. Hence, the same calculations as
  in~\eqref{eq:conv1} and ~\eqref{eq:conv2} gives the convergence
  ${\mathcal G}^{N,K}\xRightarrow{K\to\infty} {\mathcal G}^N$ as well.
  
  ~

  \noindent
  {\it Step 4: $(\mathcal G_1^{N},...,\mathcal G_n^{N})
    \xRightarrow{N\to\infty}({\mathcal G_1},...,{\mathcal G_n})$,
    constructed from infinitely many random graphs}: By now, we have
  shown that $(\mathcal G_1^{N},...,\mathcal G_n^{N})$ can be
  constructed from infinitely many random graphs $\mathcal A_N^{(0)},
  \mathcal A_N^{(1)},...$ such that $\mathcal A_N^{(0)}$ is a Kingman
  coalescent started with $n$ lines, $\mathcal A_N^{(i+1)}$ has
  additional coalescence events, {\it regular} split events at rate
  $\gamma (N-m)/(2N)$ to new lines, if there are a total of $m$ lines
  in graphs $\mathcal A_N^{(0)},...\mathcal A_N^{(i+1)}$ and {\it
    irregular} split events at rate $\gamma m/(2N)$ to already
  existing lines, if there are $m$ lines in graphs $\mathcal
  A_N^{(0)},...\mathcal A_N^{(i+1)}$. Now, as $N\to\infty$, the rate
  of regular splitting events converges to $\gamma/2$. By almost sure
  convergence of the random graphs, the genomes converge as well,
  i.e.\ $(\mathcal G_1^{N},...,\mathcal G_n^{N})
  \xRightarrow{N\to\infty}({\mathcal G_1},...,{\mathcal
    G_n})$. Precisely, we again have to check (i) and (ii) of
  Remark~\ref{rem:kall}.2. For (i), first note that (for $L(\mathcal
  A^{(0)}_{N}\cup\cdots\cup\mathcal A^{(i)}_{N})$ the total length of
  $\mathcal A_N^{(0)},...\mathcal A_N^{(i)}$)
  \begin{equation}
    \label{eq:conv5}
    \begin{aligned}
      & \mathbb P(\mathcal G_k([0,1])\geq C) \leq \frac{1}{C} \mathbb
      E[\mathcal G_k([0,1])] = \frac{\theta}{2C} \mathbb E[L(\mathcal
      A_n^{(1)})] \xrightarrow{C\to\infty} 0,\\
      & \mathbb P\Big( \mathcal A^{(i)}_{N} \text{ hit by irregular
        split event}\Big) \\ & \qquad = 1 - \mathbb E[\exp(-\gamma
      L(\mathcal A^{(0)}_{N}\cup\cdots\cup\mathcal A^{(i)}_{N})/(2N))]
      \\ & \qquad \leq \frac{(i+1)\gamma}{2N} \mathbb E[L(\mathcal
      A^{(1)}_{N})] \xrightarrow{N\to\infty} 0.
    \end{aligned}
  \end{equation}
  according to Lemma~\ref{l:length}.  
  So, we write for $A\subseteq \{1,...,n\}\times [0,1]$
  \begin{equation}
    \label{eq:conv3}
    \begin{aligned}
      |\mathbb P(& {\mathcal G}(A)=k) - \mathbb P(\mathcal
      G^{N}(A)=k)| \\ & \leq \mathbb P(\mathcal G(A)\geq C) + \mathbb
      P\Big(\bigcup_{i=1}^C
      \mathcal A^{(i)}_{N} \text{ hit by irregular split event}\Big) \\
      & \leq \mathbb P(\mathcal G(A)\geq C) + \sum_{i=1}^C \mathbb
      P(\mathcal A^{(i)}_{N} \text{ hit by irregular split event}\Big)
      \\ & \xrightarrow{N\to\infty} \mathbb P(\mathcal G(A)\geq C)
      \xrightarrow{C\to\infty} 0
    \end{aligned}
  \end{equation}
  by~\eqref{eq:conv5} implying (i) of Remark~\ref{rem:kall}.2. For
  (ii) (again noting that the same calculation holds for arbitrary
  compact $A\subseteq\{1,...,n\}\times [0,1]$), we have, since
  $(L(\mathcal A_N^{(1)}))_{N=1,2,...}$ is uniformly integrable, by
  standard arguments (see e.g.\ \cite{Billingsley1999}, Theorem~3.5),
  \begin{equation}
    \label{eq:conv4}
    \begin{aligned}
      \mathbb E[\mathcal G^{N}( & \{1,...,n\}\times [0,1])] = n\theta
      \cdot \mathbb E[ L(\mathcal A_{1,N}^{(1)})] \\ &
      \xrightarrow{N\to\infty}n\theta \cdot \mathbb E[ L(\mathcal
      A_{1}^{(1)})] = \mathbb E[\mathcal G( \{1,...,n\}\times [0,1])]<
      \infty.
    \end{aligned}
  \end{equation}

  ~

  \noindent
  {\it Step 5: Convergence of moments}: \sloppy The calculations are
  similar to~\eqref{eq:conv3} and~\eqref{eq:conv4}. We only have to
  deal with finiteness of moments in order to show $\mathcal
  G_1^N\otimes \cdots \otimes\mathcal G_n^N
  \xRightarrow{N\to\infty}\mathcal G_1\otimes \cdots \otimes \mathcal
  G_j$. Here, (i) of Remark~\ref{rem:kall}.2 is implied
  by~\eqref{eq:conv3}. For (ii), we know that $(L(\mathcal A_N^{(1)})
  \cdots L(\mathcal A_N^{(n)}))_{N=1,2,...}$ is uniformly integrable
  by Lemma~\ref{l:length} (since $L(\mathcal A_N^{(1)}) \cdots
  L(\mathcal A_N^{(n)}) \leq L(\mathcal A_N^{(1)})^n + \cdots +
  L(\mathcal A_N^{(n)})^n$ and the latter is uniformly integrable by
  Lemma~\ref{l:length}) and $L(\mathcal A_N^{(1)}) \cdots L(\mathcal
  A_N^{(n)}) \xRightarrow{N\to\infty}L(\mathcal A_1) \cdots L(\mathcal
  A_n)$. Hence, 
  \begin{equation}
    \label{eq:conv4b}
    \begin{aligned}
      \mathbb E[ & \mathcal G_1^{N}\otimes \cdots \otimes\mathcal
      G_n^{N}( (\{1,...,n\}\times [0,1]))^n] = \tfrac{\theta^n}{2^n}
      \cdot \mathbb E[ L(\mathcal A_{1,N}^{(1)})\cdots L(\mathcal
      A_{1,N}^{(n)})] \\ &
      \xrightarrow{N\to\infty}\tfrac{\theta^n}{2^n} \cdot \mathbb E[
      L(\mathcal A_1)\cdots L(\mathcal A_n)] = \mathbb E[\mathcal G_1
      \otimes \cdots \otimes \mathcal G_n( (\{1,...,n\}\times
      [0,1])^n)] \\ & < \infty.
    \end{aligned}
  \end{equation}
\end{proof}

\section{Proofs of Theorems \ref{T1}--\ref{T3}}
\label{S:proofs}

\subsection{Proof of Theorem~\ref{T1}}
Using diffusion theory and Lemma~\ref{l:diff}, we obtain first moments
of all of the statistics $G_1^{(n)},\dots,G_n^{(n)}$ in
equilibrium. Moreover, the statistics as considered in
Theorem~\ref{T2} are linear combinations of
$G_1^{(n)},\dots,G_n^{(n)}$; see the first proof of Theorem~\ref{T2}
below.

We consider the diffusion~\eqref{eq:SDE} with infinitesimal mean and
variance
\begin{align}\notag
  \mu(x) = -\tfrac \rho 2 x + \tfrac \gamma 2 x(1-x), \qquad
  \sigma^2(x) = x(1-x).
\end{align}
The Green function for the diffusion, measuring the time the
diffusion, i.e.\ a gene, spends in frequency $x$ until eventual loss,
if the current frequency is $\delta\leq x$, is given
by 
\begin{align}\notag
  G(\delta,x) = 2 \frac{\phi(\delta)}{\sigma^2(x)\psi(x)},
\end{align}
where
\begin{align*}
  \psi(y) &:= \exp\left( -2 \int_0^y \frac{\mu(z)}{\sigma^2(z)} dz
  \right)
  = (1-y)^{-\rho} e^{-\gamma y},\\
  \phi(x) &:= \int_0^x \psi(y) dy.
\end{align*}
Following \cite[chapt. 7.11]{Durrett2008}, we introduce new genes in
frequency $\delta\ll 1$ at rate $\frac{\theta}{2}
\frac{1}{\phi(\delta)}$ in a consistent way. That is, the gene rises
in frequency to $\varepsilon>\delta$ with probability
$\frac{\phi(\delta)}{\phi(\epsilon)}$. Hence the number of genes in
frequency $x$ is Poisson with mean
\begin{align}\notag
  \frac{\theta}{2} \frac{1}{\phi(\delta)} G(\delta,x) = \theta
  \frac{e^{\gamma x}}{x(1-x)^{1-\rho}}.
\end{align}
The gene frequency spectrum is now given by
\begin{align*}
  \mathbb E[ G_k^{(n)} ] &= \binom{n}{k} \int_0^1 \theta
  \frac{e^{\gamma x}}{x(1-x)^{1-\rho}} x^k (1-x)^{n-k} dx \\ &=
  \binom{n}{k} \theta \int_0^1 e^{\gamma x} x^{k-1} (1-x)^{n-k-1+\rho}
  dx \\ &= \theta \binom{n}{k} (k-1)!
  \frac{\Gamma(n-k+\rho)}{\Gamma(n+\rho)} {}_1F_1(k;n+\rho;\gamma) \\
  & = \frac{\theta}{k} \frac{(n)_{
      k \downarrow}}{(n-1+\rho)_{k \downarrow}} \Big(1 +
  \sum\limits_{m=1}^{\infty} \frac{(k)_{m \uparrow}
    \gamma^m}{(n+\rho)_{m \uparrow} m!}\Big)
\end{align*}
where ${}_1F_1(k;n+\rho;\gamma) = 1 + \sum\limits_{m=1}^{\infty}
\frac{(k)_{m \uparrow} \gamma^m}{(n+\rho)_{m \uparrow} m!}$ is a confluent
hypergeometric function (Kummer's function), see chapter 13 in \cite{NIST2010}.

\subsection{Proof of Theorem~\ref{T2}}
We give two proofs, one using diffusion theory and Theorem~\ref{T1},
one using the AGTG from Section~\ref{S:AGTG}.

\begin{proof}[Proof of Theorem~\ref{T2} using Theorem~\ref{T1}]
  Given the expected gene frequency spectrum from Theorem~\ref{T1}, it
  is now easy to compute first moments of $A$, $D$ and $G$ by using,
  in the infinite population limit,
  \begin{equation}
    \label{eq:429}
    \begin{aligned}
      A^{(1)} & \stackrel d = G_1^{(1)}, \qquad D^{(2)} \stackrel d = \tfrac 12 G_1^{(2)}, \\
      G^{(n)} & = |\mathcal G_1| + |\mathcal G_2\setminus \mathcal G_1|
      + \cdots + \Big| \mathcal G_n \setminus \bigcup_{i=1}^{n-1}
      \mathcal G_{i}\Big|
    \end{aligned}
  \end{equation}
  such that
  \begin{align*}
    \mathbb E[A^{(n)}] & = \mathbb E[A^{(1)}] = \mathbb E[G_1^{(1)}] =
    \frac{\theta}{\rho} \left( 1 +
      \sum\limits_{m=1}^{\infty} \frac{\gamma^m}{(1+\rho)_{m \uparrow}} \right), \\
    \mathbb E[D^{(n)}] & = \mathbb E[D^{(2)}] = \tfrac 12 \mathbb
    E[G_1^{(2)}] = \frac{\theta}{1 + \rho} \left( 1 +
      \sum\limits_{m=1}^{\infty} \frac{\gamma^m}{(2+\rho)_{m \uparrow} } \right), \\
    \mathbb E[G^{(n)}] & = \sum_{k=1}^n \tfrac 1k \mathbb E[G_1^{(k)}] =
    \sum_{k=1}^n \frac{\theta}{k} \frac{k}{k-1+\rho}\sum_{m=0}^\infty
    \frac{\gamma^m}{(k+\rho)_{m \uparrow}} \\ & = \theta
    \sum_{m=0}^\infty {\gamma^{m}} \sum_{k=0}^{n-1}
    \frac{1}{(k+\rho)_{{m+1}\uparrow}} \\ & = \theta
    \sum_{k=0}^{n-1}\frac{1}{k+\rho} + \theta\sum_{m=1}^\infty
    \frac{\gamma^{m}}{m} \sum_{k=0}^{n-1}\Big( \frac{1}{(k+\rho)_{m
        \uparrow}} - \frac{1}{(k+1+\rho)_{m \uparrow}}\Big) \\ &
    =\theta \sum_{k=0}^{n-1}\frac{1}{k+\rho} + \theta\sum_{m=1}^\infty
    \frac{\gamma^{m}}{m} \Big(\frac{1}{(\rho)_{m \uparrow}} -
    \frac{1}{(n+\rho)_{m \uparrow}}\Big).
  \end{align*}
\end{proof}

\begin{proof}[Proof of Theorem~\ref{T2} using the AGTG]
  First we note that $A^{(1)} = G^{(1)}$ and $D^{(2)} = \tfrac 12
  (|(\mathcal G_1^{(1)} \cup\mathcal G_1^{(2)}) \setminus \mathcal
  G^{(1)}| + |(\mathcal G_1^{(1)} \cup \mathcal G_1^{(2)}) \setminus
  \mathcal G^{(2)}|)$ such that
  \begin{align*}
    \mathbb E[A^{(n)}] & = \mathbb E[A^{(1)}] = \mathbb E[G^{(1)}],\\
    \mathbb E[D^{(n)}] & = \mathbb E[D^{(2)}] = \mathbb E[G^{(2)}] -
    \mathbb E[G^{(1)}],
  \end{align*}
  and it suffices to compute $\mathbb E[G^{(n)}]$ in the proof. We
  will abuse notation and write $dx$ and $dy$ for \emph{infinitely
    small} portions of the genome. In order to compute $\mathbb
  E[G^{(n)}]$, the idea is to write $\mathcal G^{(n)} := (\sum_{i=1}^n
  \mathcal G_i)\wedge 1$ and
  \begin{equation}
    \label{eq:431}
    \begin{aligned}
      \mathbb E[G^{(n)}] = \mathbb E\Big[ \int_0^1 \mathcal
      G^{(n)}(dx)\Big] = \int_0^1 \mathbb E[\mathcal G^{(n)}(dx)] = \int_0^1
      \frac\theta 2 \mathbb E[L(\mathcal A_n)] dx = \frac\theta 2
      \mathbb E[L(\mathcal A_n)],
    \end{aligned}
  \end{equation}
  such that we have to compute the expected length of $\mathcal A_n$,
  the AGTG for a single gene, which we denote by $L(\mathcal A_n)$.
  Therefore consider the birth and death process 
   $(Z_t)_{t\geq 0}$ with birth rate $\lambda_i = \gamma$ and death
  rate $\mu_i =i+\rho-1$.
  Recall that the hitting time $T$, when this birth and death process hits zero has the same distribution as
  $L(\mathcal A_n)/2$, see proof of Lemma~\ref{l:length}.
%
%
%
  Now, it is well known (see e.g.\ \citealp[chapt. 4.7]{KarlinTaylor1975})
  that
  \begin{equation}
    \label{eq:432}
    \mathbb E[L(\mathcal A_n)] = \mathbb E[T|Z_0=n] = 
    \sum_{i=1}^\infty p_i + \sum_{k=1}^{n-1} \left( \prod_{r=1}^{k} \frac{\mu_r}{\lambda_r} 
    \right) \sum_{m=k+1}^\infty  p_m
  \end{equation}
  where
  \begin{equation}
    p_i = \frac{\lambda_1 \cdots \lambda_{i-1}}
    {\mu_1 \cdots \mu_i} = \frac{\gamma^{i-1}}{\rho(1+\rho)\cdots (i-1+\rho)} = 
    \frac{\gamma^{i-1}}{(\rho)_{i \uparrow}}.
  \end{equation}
  Combining~\eqref{eq:431} and~\eqref{eq:432} yields
  \begin{equation}
    \label{eq:912}
    \begin{aligned}
      \frac{1}{\theta} \mathbb E[G^{(n)}] & = \frac{1}{2} \mathbb
      E[L(\mathcal A_n)] 
      \\ & = \sum_{i=1}^\infty \frac{\gamma^{i-1}}{(\rho)_{i
          \uparrow}} + \sum_{k=1}^{n-1}
      \frac{(\rho)_{k \uparrow}}{\gamma^k} \sum_{m=k+1}^\infty \frac{\gamma^{m-1}}{(\rho)_{m \uparrow}}\\
      &= \sum_{k=1}^{n-1}\sum_{m=k+1}^\infty
      \frac{\gamma^{m-1-k}}{(\rho+k)_{(m-k)\uparrow}} +
      \sum_{i=1}^\infty \frac{\gamma^{i-1}}{(\rho)_{i \uparrow}}
      \\ &= \sum_{m=1}^\infty \sum_{k=0}^{n-1} \frac{\gamma^{m-1}}{(\rho+k)_{m \uparrow}}\\
      &= \sum_{k=0}^{n-1} \frac {1}{k+\rho} + \sum_{m=1}^\infty
      \gamma^m \sum_{k=0}^{n-1}
      \frac{1}{(\rho+k)_{(m+1)\uparrow} } \\
      &= \sum_{k=0}^{n-1} \frac {1}{k+\rho} + \sum_{m=1}^\infty
      \frac{\gamma^m}{m}
      \sum_{k=0}^{n-1}  \left( \frac{1}{(\rho+k)_{m \uparrow}} - \frac{1}{(\rho+k+1)_{m \uparrow}} \right)\\
      &= \sum_{k=0}^{n-1} \frac {1}{k+\rho} + \sum_{m=1}^\infty
      \frac{\gamma^m}{m} \left( \frac{1}{(\rho)_{m \uparrow}} -
        \frac{1}{(n+\rho)_{m \uparrow}} \right).
    \end{aligned}
  \end{equation}
  According to \eqref{eq:429}, $\mathbb E[A^{(n)}]$ is readily
  obtained and the expected number of differences is given
  using~\eqref{eq:912} by
  \begin{align*}
    \frac{1}{\theta}\mathbb E[D^{(n)}] &= \frac{1}{\theta}\big(\mathbb
    E[G^{(2)}] - \mathbb E[G^{(1)}]\big)
    = \sum_{m=1}^\infty \frac{\gamma^{m-1}}{(\rho+1)_{m \uparrow}}\\
    &= \sum_{m=0}^\infty \frac{\gamma^{m}}{(1+\rho)_{(m+1)\uparrow}}
    = \frac{1}{1+\rho} \left( 1 + \sum_{m=1}^\infty
      \frac{\gamma^m}{(2+\rho)_{m \uparrow}} \right).
  \end{align*}
\end{proof}

\subsection{Proof of Theorem~\ref{T3}}
Since $A^{(1)} = \mathcal G_1([0,1]) = \int_0^1 \mathcal G_1(dx)$, we
can use the first and second moment measures of $\mathcal G_1$ in
order to compute the moments of $A^{(1)}$; see e.g.\ \cite{Daley},
Section~5.4, which are given by $A\mapsto \mathbb E[\mathcal G_1(A)]$
for the first and $(A,B)\mapsto \mathbb E[\mathcal G_1(A)\mathcal
G_1(B)]$ for the second moment. (A similar statement holds for the random measure $\mathcal
D_{1,2} := |\mathcal G_1^N - \mathcal G_2^N|$ and $2D^{(2)} = \mathcal
D_{1,2}([0,1])$.) For the integral with respect to these measures, we
will -- as in~\eqref{eq:431} -- abuse notation (see e.g.\ the term
$\mathbb V[\mathcal G_1(dx)]$ below) such that
\begin{align}\label{eq:233b}
  \mathbb V[A^{(1)}] & = \int_0^1 \mathbb V[\mathcal G_1(dx)] +
  \int_0^1 \int_0^1 1_{x\neq y}\mathbb{COV}[\mathcal G_1(dx), \mathcal
  G_1(dy)].
\end{align}
First, given $\mathcal A^{(1)}_1$ (which is distributed like the AGTG for a single gene $\mathcal A_1$), the
probability of a gene gain event on $\mathcal A^{(1)}_1$ is
$(\theta/2)L(\mathcal A^{(1)}_1) dx$ such that
\begin{equation}
  \label{eq:234}
  \begin{aligned}
    \mathbb V[|\mathcal G_1(dx)|] &= \mathbb V\big[\mathbb E[
    |\mathcal G_1(dx)|\,|\mathcal A^{(1)}_1]\big] + \mathbb E\big[
    \mathbb V[|\mathcal G_1(dx)||\mathcal A^{(1)}_1] \big]\\
    &= \mathbb V\big[ \tfrac{\theta}{2} L(\mathcal A^{(1)}_1) dx \big] +
    \mathbb E \big[ \mathbb E[|\mathcal G_1(dx)|\, |\mathcal
    A^{(1)}_1]\big]\\
    & = \frac{\theta}{2} \mathbb E[L(\mathcal A_1)] dx + \mathcal
    O(dx^2) = \mathbb E[A^{(1)}]dx + \mathcal O(dx^2) \\ & =
    \frac{\theta}\rho\Big( 1 + \frac{\gamma(2+\rho) +
      \gamma^2}{(1+\rho)(2+\rho)} + \mathcal O(\gamma^3)\Big) dx+
    \mathcal O(dx^2)
  \end{aligned}
\end{equation}
Second, for $x\neq y$,
\begin{align*}
  \mathbb{COV}[|\mathcal G_1(dx)|,|\mathcal G_1(dy)|] &=
  \mathbb{COV}\Big[\mathbb E[|\mathcal G_1(dx)|\,|\mathcal
  A^{(1)}_1,\mathcal A^{(2)}_1], \mathbb E[|\mathcal G_1(dy)|\,|\mathcal
  A^{(1)}_1,\mathcal A^{(2)}_1]\Big] \\ & \qquad \qquad \qquad + \mathbb E\Big[
  \mathbb{COV}[|\mathcal G_1(dx)|, |\mathcal G_1(dy)| \,
  |\mathcal A^{(1)}_1,\mathcal A^{(2)}_1] \Big]\\
  &= \mathbb{COV}\big[ \tfrac{\theta}{2} L(\mathcal A^{(1)}_1) dx,
  \tfrac{\theta}{2} L(\mathcal A^{(2)}_1) dy \big] \\
  & = \frac{\theta^2}{4} \mathbb{COV}[L(\mathcal A^{(1)}_1),L(\mathcal A^{(2)}_1)]
  dx\, dy
\end{align*}
since $|\mathcal G_1(dx)|$ and $|\mathcal G_1(dy)|$ are independent
given $\mathcal A^{(1)}_1,\mathcal A^{(2)}_1$. Now we compute the term 
$\mathbb{COV}[L(\mathcal A^{(1)}_1),L(\mathcal A^{(2)}_1)]$ up to second order in
$\gamma$. For this computation, we make use of the fact that the AGTG
for two genes can be defined in analogy to the AGTG for a single gene
from Definition~\ref{def:AGTGsingle}, but with two different kind of
loss and transfer events. Precisely, we consider the following random
graph: starting with $x$ lines of state \emph{only gene 1}, $y$ lines
of state \emph{both genes} and $z$ lines \emph{only gene 2}, pairs of
lines coalesce at rate~1. (Note that coalescence of a line of state
\emph{only gene 1} and a line of state \emph{only gene 2} gives a
single line of state \emph{both genes}.)  Lines where gene~1 (gene 2)
is considered are lost at rate $\rho/2$. (If a line of state
\emph{only gene 1} (\emph{only gene 2)} is lost, it is lost
completely, while if a line of state \emph{both genes} is lost, it
turns into a line of state \emph{only gene 2} (\emph{only gene 1}).)
Finally, every line of state \emph{only gene 1} (\emph{only gene 2})
is split at rate $\gamma/2$ and the new line is again of state
\emph{only gene 1} (\emph{only gene 2}). In addition, a line of state
\emph{both genes} splits at rate $\gamma$ and the new line is of state
\emph{only gene 1} or \emph{only gene 2}, both with probability 1/2.
The length of the graph of lines at states with gene 1 (gene 2),
i.e. either at state \emph{only gene 1} (\emph{only gene 2}) or
\emph{both genes} is denoted $L_1(t)$ ($L_2(t)$) if a sample from time
$t$ of the population is considered.  We write $\mathbb E_{xyz}[.]$
for the expected value if the process is started as above.

It is important to note that $\mathbb E[L(\mathcal A^{(1)}_1)L(\mathcal
A^{(2)}_1)] = \mathbb E_{010}[L_1(t)L_2(t)]$ for any $t$, since the AGTG
describes the population in equilibrium. In order to compute $\mathbb
E_{010}[L_1(t) L_2(t)]$, we use a time derivative and write
\begin{align*}
  \mathbb E_{010}[L_1(t+dt)L_2(t+dt)] & = (1-(\gamma+\rho)dt) \mathbb
  E_{010}[(L_1(t)+dt)(L_2(t)+dt)] \\ & \qquad + \gamma dt\cdot \mathbb
  E_{110}[(L_1(t)+dt)(L_2(t)+dt)] + \rho dt\cdot 0
\end{align*}
Using that the AGTG is in equilibrium and ignoring effects of order
$dt^2$, we obtain (for $L_i :=L_i(t)$, $i=1,2$)
\begin{equation}
  \label{eq:lin1}
  \begin{aligned}
    (\gamma + \rho) \mathbb E_{010}[L_1L_2] & = \mathbb E_{010}[L_1 +
    L_2] + \gamma\cdot \mathbb E_{110}[L_1L_2],\\
    (1 + \tfrac 32 \rho + \tfrac 32 \gamma)\mathbb E_{110}[L_1L_2] & =
    \mathbb E_{110}[L_1+2L_2] + \gamma \cdot \mathbb E_{210}[L_1L_2] +
    \tfrac 12 \gamma \mathbb E_{111}[L_1L_2] \\ & \qquad \qquad \quad
    + \tfrac 12 \rho \cdot
    \mathbb E_{101}[L_1L_2] + (1+\tfrac 12 \rho) \cdot \mathbb E_{010}[L_1L_2],\\
    (3 + 2\rho)\mathbb E_{210}[L_1L_2] & = \mathbb E_{210}[L_1+3L_2] +
    (3+ \rho)\cdot \mathbb E_{110}[L_1L_2] \\ & \qquad \qquad \qquad
    \qquad +
    \tfrac 12 \rho \cdot \mathbb E_{201}[L_1L_2] + \mathcal O(\gamma),\\
    (3 + 2\rho)\mathbb E_{111}[L_1L_2] & = \mathbb E_{111}[2L_1 +
    2L_2] + \mathbb E_{020}[L_1L_2] \\ & \qquad + (2+\rho)\cdot
    \mathbb E_{110}[L_1L_2] + \rho\cdot \mathbb E_{201}[L_1L_2] +
    \mathcal
    O(\gamma),\\
    (1+\gamma+\rho)\mathbb E_{101}[L_1L_2] & = \mathbb E_{101}[L_1 +
    L_2] + \mathbb E_{010}[L_1L_2] + \gamma\cdot \mathbb E_{201}[L_1L_2],\\
    (3 + \tfrac 32\rho)\mathbb E_{201}[L_1L_2] & = \mathbb E_{201}[L_1
    + 2L_2] + (1 + \rho) \cdot \mathbb E_{101}[L_1L_2] \\ & \qquad
    \qquad \qquad \qquad \qquad + 2\cdot \mathbb E_{110}[L_1L_2] +
    \mathcal O(\gamma),\\
    (1 + 2\rho)\mathbb E_{020}[L_1L_2] & = \mathbb E_{020}[2L_1 +
    2L_2] + \mathbb E_{010}[L_1L_2] \\ & \qquad \qquad \qquad \qquad
    \qquad + 2\rho \cdot \mathbb E_{110}[L_1L_2] + \mathcal O(\gamma).
  \end{aligned}
\end{equation}
Note that some terms $\mathcal O(\gamma)$ were written which will not
lead to the first two leading terms in $\mathbb
E_{010}[L_1(t)L_2(t)]$. The expectations $\mathbb E_j[L_i]$ for
$i=1,2$ and $j\in\{010,110,101,210,111,201\}$ can readily be computed
using the AGTG for a single gene, since
\begin{align}
  \label{eq:lin2}
  \mathbb E_{xyz}[L_1] = \mathbb E[L(\mathcal A_{x+y})]\text{ and
  }\mathbb E_{xyz}[L_2] = \mathbb E[L(\mathcal A_{y+z})].
\end{align}
We use from~\eqref{eq:912} that
\begin{align*}
  \mathbb E[L(\mathcal A_n)] & = \sum_{k=0}^{n-1} \frac{2}{k+\rho} +
  2\gamma \frac{n}{\rho(n+\rho)} + \gamma^2 \frac{n(n +
    2\rho+1)}{\rho(\rho+1)(n+\rho)(n+\rho+1)} + \mathcal O(\gamma^3),
\end{align*}
such that
\begin{equation}
  \label{eq:lin3}
  \begin{aligned}
    \mathbb E[L(\mathcal A_1)] & = \frac 2\rho\Big(1 +
    \frac{\gamma}{1+\rho}\Big) + \gamma^2
    \frac{2}{\rho(\rho+1)(\rho+2)}
    + \mathcal O(\gamma^3),\\
    \mathbb E[L(\mathcal A_2)] & = \frac {2+4\rho}{\rho(\rho+1)} + \frac{4\gamma}{\rho(\rho+2)}  + \mathcal O(\gamma^2) ,\\
    \mathbb E[L(\mathcal A_3)] & = \frac {6\rho^2 + 12\rho +
      4}{\rho(\rho+1)(\rho+2)} + \mathcal O(\gamma).
  \end{aligned}
\end{equation}
Solving~\eqref{eq:lin1} using~\eqref{eq:lin2} and~\eqref{eq:lin3}
gives
\begin{equation}
  \label{eq:lin4}
  \begin{aligned}
    \mathbb{COV}[L(\mathcal A^{(1)}_1), L(\mathcal A^{(2)}_1)] & = \mathbb
    E_{010}[L_1L_2] - \mathbb E[L(\mathcal A_1)]^2 \\ & =
    \frac{4}{\rho(1+\rho)^2(3+2\rho)(2+7\rho+6\rho^2)}\gamma^2 +
    \mathcal O(\gamma^3).
  \end{aligned}
\end{equation}
Combining~\eqref{eq:lin4} with~\eqref{eq:234} and~\eqref{eq:233b}
gives the result.

~

\noindent
To compute the variance for the number of differences $D^{(2)}$ we
will use a similar approach. Recall $\mathcal D_{1,2} = |\mathcal
G_1^N - \mathcal G_2^N|$ and $2D^{(2)} = \int_0^1 \mathcal
D_{1,2}(dx)$. Thus, 
\begin{equation}
  \label{vd1}
  \mathbb V[2D^{(2)}]  = \int_0^1 \mathbb V[\mathcal D_{1,2}(dx)] +
  \int_0^1 \int_0^1 1_{x\neq y}\mathbb{COV}[\mathcal D_{1,2}(dx), \mathcal
  D_{1,2}(dy)].
\end{equation}
Let $\mathcal A^{(i,\text{sing})}_2$ be the subgraph of $\mathcal
A_2^{(i)}$ consisting of branches leading to either individual~1 or
individual~2 but not to both. If $L(\mathcal A^{(i,\text{sing})}_2)$
denotes its length, given $\mathcal A_2^{(i)}$, the chance of a gene
gain event in $dx$ leading to a difference between two given
individuals is $(\theta/2) L(\mathcal A^{(i,\text{sing})}_2)dx$ such
that (compare with \eqref{eq:234})
\begin{equation}
  \label{eq:vd2}
  \begin{aligned}
    \mathbb V[\mathcal D_{1,2}(dx)] &= \mathbb V\big[\mathbb E[
    \mathcal D_{1,2}(dx)|\mathcal A^{(i)}_2]\big] + \mathbb E\big[
    \mathbb V[\mathcal D_{1,2}(dx)|\mathcal A^{(i)}_2] \big]\\
    &= \mathbb V\big[ \tfrac{\theta}{2} L(\mathcal
    A^{(i,\text{sing})}_2) dx \big] + \mathbb E \big[ \mathbb
    E[\mathcal D_{1,2}(dx) |\mathcal
    A^{(i)}_2]\big]\\
    & = \frac{\theta}{2} \mathbb E[L(\mathcal A^{(i,\text{sing})}_2)]
    dx + \mathcal O(dx^2) = \frac{\theta}{2} \mathbb E[2D^{(2)}]dx +
    \mathcal O(dx^2) \\ & = \theta \frac{2}{1 + \rho} + \frac{2
      \gamma}{ 2 + 3 \rho + \rho^2} + \mathcal O(\gamma^2) + \mathcal
    O(dx^2).
  \end{aligned}
\end{equation}

In the same way as seen below equation \eqref{eq:234} we obtain, for
$i\neq j$
\begin{align*}
  \mathbb{COV}[\mathcal D_{1,2}(dx),\mathcal D_{1,2}(dy)] & =
  \frac{\theta^2}{4} \mathbb{COV}[L(\mathcal A^{(i,\text{sing})}_2),L(\mathcal A^{(j,\text{sing})}_2)]
  dx\, dy.
\end{align*}
As $\mathbb E[L(\mathcal A^{(i,\text{sing})}_2)]\cdot \mathbb
E[L(\mathcal A^{(j,\text{sing})}_2)]$ is already known the remaining
part is to compute
\begin{equation}
  \label{eq:vd3}
  \begin{aligned}
    \mathbb E[L(\mathcal A^{(i,\text{sing})}_2) L(\mathcal A^{(j,\text{sing})}_2)] &=
    \frac{32} {(1 + \rho) (1 + 2\rho)} \\
    & \quad + \frac{ 32 (48 + 314 \rho + 611 \rho^2 + 464 \rho^3 + 120
      \rho^4) \gamma} {(1 + \rho) (2 + \rho) (1 + 2 \rho)^2 (3 + 2
      \rho) (2 + 3 \rho) (6 + 5 \rho)} + \mathcal O(\gamma^2).
  \end{aligned}
\end{equation}
For that we will split $(\mathcal A_2^{(i)},\mathcal A_2^{(j)})$ into
two parts, $T(\mathcal A_2^{(i)},\mathcal A_2^{(j)})$ and $S(\mathcal
A_2^{(i)},\mathcal A_2^{(j)})$. Recall that there are three different
types of events in $(\mathcal A_2^{(i)},\mathcal A_2^{(j)}) $, namely
loss, merging lines and splitting lines. The first part, $ T(\mathcal
A_2^{(i)},\mathcal A_2^{(j)})$, contains solely the times
$T_1,T_2,\dots$ between these events, while the second part,
$S(\mathcal A_2^{(i)},\mathcal A_2^{(j)})$ contains the remaining
information from $(\mathcal A^1,\mathcal A^2)$ on which lines split,
merge and get lost, i.e.\ it is possible to describe the
structure/topology/shape of the AGTG from $S(\mathcal
A_2^{(i)},\mathcal A_2^{(j)})$. Note that given $S(\mathcal
A_2^{(i)},\mathcal A_2^{(j)})$, the times $T_1 = T_1(S(\mathcal
A_2^{(i)},\mathcal A_2^{(j)})), T_2 = T_2(S(\mathcal
A_2^{(i)},\mathcal A_2^{(j)})),...$ are independent exponentially
distributed random variables with rates measurable with respect to
$S(\mathcal A_2^{(i)},\mathcal A_2^{(j)})$. In particular, the number
of lines between the $k$th and $(k+1)$st time in $T(\mathcal
A_2^{(i)},\mathcal A_2^{(j)})$, which lead to either one or the other
of the individuals, but not to both, denoted by $D_k^i =
D_k^i(S(\mathcal A_2^{(i)},\mathcal A_2^{(j)}))$, is $S(\mathcal
A_2^{(i)},\mathcal A_2^{(j)})$-measurable and
\begin{equation}
  L(\mathcal A^{(i,\text{sing})}_2) =  \sum_k D_k^i T_k
\end{equation}
Let $\mathcal S$ be the space of all possible shapes which can be
taken by $S(\mathcal A_2^{(i)},\mathcal A_2^{(j)}) $ and let $\mathcal
S_{\gamma^2} := \{s \in \mathcal S : \mathbb P(s) \notin \mathcal
O(\gamma^2) \}$, i.e.\ $ \mathcal S_{\gamma^2}$ contains all shapes
which have at most one splitting event. Within $\mathcal
S_{\gamma^2}$, there are at most 8 events before $(\mathcal
A_2^{(i)},\mathcal A_2^{(j)})$ has lost all lines, so we can write
\begin{align*}
  \mathbb E[L(\mathcal A_2^{(i,\text{sing})}) & L(\mathcal A_2^{(j,\text{sing})})] = \mathbb E
  [ \mathbb E [ L(\mathcal A_2^{(i,\text{sing})}) L(\mathcal A_2^{(j,\text{sing})})] |
  S(\mathcal A_2^{(i)},\mathcal A_2^{(j)}) ] ] \\ &= \sum_{s \in
    \mathcal S_{\gamma^2}} \mathbb P(s) \cdot \mathbb E [ L(\mathcal
  A_2^{(i,\text{sing})}) L(\mathcal A_2^{(j,\text{sing})})] | S(\mathcal A_2^{(i)},\mathcal
  A_2^{(j)})=s ]
  + \mathcal O (\gamma^2) \\
  &= \sum_{s \in \mathcal S_{\gamma^2}} \mathbb P(s) \cdot \mathbb
  E\Big[ \sum_{k=1}^8 D_k^i T_k \sum_{k=1}^8 D_k^j T_k|S(\mathcal
  A_2^{(i)},\mathcal A_2^{(j)})=s \Big] + \mathcal O (\gamma^2) \\
  &= \sum_{s \in \mathcal S_{\gamma^2}} \mathbb P(s) \sum_{k=1}^8
  D_k^i(s) D_k^j(s)
  \mathbb E[T_k^2(s)] \\
  & \quad + \mathbb P(s) \sum_{1\leq k, k' \leq 8 ; k \neq k'}
  D_k^i(s) D_{k'}^j(s) \mathbb E[T_k(s)] \mathbb E[T_{k'}(s)] +
  \mathcal O (\gamma^2)
\end{align*}
As $\mathcal S_{\gamma^2}$ has more than $5000$ elements we used
Mathematica to compute
$\mathbb P(s)$ -- see the accompanying file available at the arXiv (\url{http://arxiv.org/abs/1301.6547}) -- the variables
$D_k^i(s)$, resp. $D_k^j(s)$, and the parameters of the exponentially
distributed times $T_k(s)$ for $1 \leq k \leq 8$ and all $s \in
\mathcal S_{\gamma^2}$.  Combining \eqref{eq:vd2} and \eqref{eq:vd3}
gives the result as shown in \eqref{eq:T3e}.

\subsubsection*{Acknowledgments}
We thank Wolfgang Hess for fruitful discussions. The DFG is
acknowledged for funding via the project PP672/2-1. The SPP 1590
funded by the DFG is acknowledged for travel support.

\bibliographystyle{chicago}

\end{document}